%% file: rng1-reconciled.tex
\documentclass{amsart}
\usepackage{amsmath,amssymb,amsthm,MnSymbol,color,nicefrac,tikz,xcolor,colortbl}
\usepackage[top=1.2in,bottom=1.2in,left=1.2in,right=1.2in]{geometry}
\usepackage{pdflscape,pgfplots}
\title{Random nilpotent groups I}
\author[Cordes Duchin Duong Ho S\'anchez]{%
Matthew Cordes$^\clubsuit$, 
Moon Duchin$^\diamondsuit$,
Yen Duong$^\heartsuit$, 
Meng-Che Ho$^\spadesuit$,\\
and Andrew P. S\'anchez$^\diamondsuit$}
\thanks{
$\clubsuit$ {\sc Technion - IIT}, Haifa, Israel,
$\diamondsuit$ {\sc Tufts University}, Medford MA, USA,   
$\heartsuit$ {\sc University of Illinois at Chicago}, Chicago IL, USA,
$\spadesuit$ {\sc University of Wisconsin at Madison}, Madison WI, USA
 \\
Correspondence: {\sf moon.duchin@tufts.edu}}

\newcommand{\Z}{\mathbb{Z}}
\newcommand{\R}{\mathbb{R}}

\newcommand{\A}{\mathbf{A}}
\newcommand{\B}{\mathbf{B}}

\DeclareMathOperator{\rank}{rank}
\DeclareMathOperator{\step}{step}
\DeclareMathOperator{\spn}{span}
\newcommand{\ab}{\mathrm{ab}}
\newcommand{\nc}[1]{\llangle #1 \rrangle}
\newcommand{\pres}[2]{\langle \, #1 \mid #2 \, \rangle}
\newcommand{\aas}{a.a.s.\ }
\newcommand{\sdot}{\! \cdot \!}
\DeclareMathOperator{\MB}{{MB}}
\DeclareMathOperator{\NB}{{NB}}
\newcommand{\groupmod}[2]{\raisebox{.02in}{$#1$} \big{/} \raisebox{-.05in}{$#2$}}

\newcommand{\vv}[2]{   \left(\begin{smallmatrix} #1 \\ #2 \end{smallmatrix}\right) }

\newcommand{\mm}[4]{   \left[ \begin{smallmatrix} #1 & #2 \\ #3 &#4 \end{smallmatrix} \right] }

\theoremstyle{definition}
\newtheorem{theorem}{Theorem}
\newtheorem{prop}[theorem]{Proposition}
\newtheorem{proposition}[theorem]{Proposition}
\newtheorem{lemma}[theorem]{Lemma}
\newtheorem{corollary}[theorem]{Corollary}
\newtheorem{remark}[theorem]{Remark}

\newtheorem{app-lemma}{Lemma}[section]

\newcommand{\PP}{{\mathbf P}}
\newcommand{\Pbar}{\overline{\Pr}}
\newcommand{\Psm}{\mathcal P_1}
\newcommand{\Pmed}{\mathcal P_2}
\newcommand{\Plg}{\mathcal P_3}
\newcommand{\Phuge}{\mathcal P_4}

\begin{document}

\begin{abstract}
We study random nilpotent groups in the well-established style of random groups,
by choosing relators uniformly among freely reduced words of (nearly) equal length 
and letting the length tend to infinity.
Whereas random groups $\Gamma=F_m/\nc R$
are quotients of a free group 
by such a random set of relators, {\em random nilpotent groups}
are formed as corresponding
quotients $G=N_{s,m}/\nc R$ of a free nilpotent group.

Using arithmetic uniformity for the random walk on $\Z^m$ and group-theoretic
results relating a nilpotent group to its abelianization, we are able to deduce 
statements about the distribution of ranks for random nilpotent groups from the literature
on random lattices and random matrices.  
We obtain results about the distribution of group orders for some finite-order cases 
as well as the probability that random nilpotent groups are abelian.
For example, for  balanced presentations (number of relators
equal to number of generators), the probability that a random nilpotent
group is abelian can be calculated for each rank $m$, 
and approaches  $84.69...\%$ as $m\to\infty$. Further, abelian implies cyclic in this setting
(asymptotically almost surely).

Considering the abelianization also yields the precise vanishing threshold for random nilpotent groups---the 
analog of the famous density one-half theorem for random groups.
A random nilpotent group is trivial if and only if the corresponding random group
is {\em perfect}, i.e., is equal to its commutator subgroup,
so this gives a precise threshold at which random groups are perfect.
More generally, we describe how to lift results about random nilpotent groups to obtain information about the lower central series of 
standard random groups.
\end{abstract}

\maketitle

\section{Introduction and background}
\subsection{Random groups}
The background idea for the paper is the models of random groups $\Gamma=F_m/\nc R$, where $F_m$ is the free group
on some number $m$ of generators, and $R$ is a set of relators of length $\ell$ chosen by a random process.
Typically one takes the number of relators $|R|$ to be a function of $\ell$; for fixed $\ell$, there are finitely many choices of $R$ of a certain size, and they are all made equally likely.
For instance, in the {\em few-relators model}, $|R|$ is a fixed constant, and in the standard
{\em density model}, $|R|=(2m-1)^{d\ell}$ for
a density parameter $0<d<1$.
(When the number of relators has sub-exponential growth, this is often regarded as sitting in the density 
model at density zero.)

After fixing $|R|$ as a function of $\ell$, 
we can write $\Pr(\Gamma~\text{has property } P)=p$ to mean that the proportion of such presentations for which the group has $P$
tends to $p$ as $\ell\to\infty$.  In particular, we say that random groups have $P$ asymptotically
almost surely (a.a.s.) if the probability tends to $1$.

The central result in the study of random groups is the theorem of Gromov--Ollivier stating that for $d>1/2$ in the density model,
$\Gamma$ is \aas isomorphic to either $\{1\}$ or $\Z/2\Z$ (depending on the parity of $\ell$), while for $d<1/2$, 
$\Gamma$ is \aas non-elementary hyperbolic and torsion-free \cite[Thm 11]{Ol}.
In the rest of this paper, we will choose our relators from those of length $\ell$ and $\ell-1$ with equal probability
in order to avoid the parity issue; with this convention, $\Gamma\cong\{1\}$ \aas for $d>1/2$. 

The Gromov--Ollivier theorem tells us that 
the density threshold for trivializing a free group coincides with the threshold for hyperbolicity, which means that
one never sees other kinds of groups, for example abelian groups, in this model. Indeed, because
$\Z^2$ can not appear as a subgroup of a hyperbolic group, one never sees a group with even one pair of 
commuting elements.
To be precise, all finitely-generated groups are quotients of $F_m$, but probability of getting 
a nontrivial, non-hyperbolic group (or a group with torsion) is asymptotically zero at every density $d\neq 1/2$. 
 Furthermore
the recent paper \cite{d12} shows that this trivial/hyperbolic dichotomy seems to persist even at $d=1/2$.

However, it is a simple matter to create new models of random groups by starting with a different
``seed" group in place of the free group $F_m$.  
The $r$ random strings in $\{a_1,\dots,a_m\}$ that are taken as relators in the Gromov model can be interpreted
as elements of any other group with $m$ generators.  
For instance, forming random quotients of the free abelian group $\Z^m$ in this way would produce a model of 
random abelian groups; equivalently, the random groups arise as cokernels of random $m\times r$ integer matrices
with columns given by the Gromov process,
and these clearly recover the abelianizations of Gromov random groups.
Random abelian groups are relatively well-studied, and information pertaining to their rank
distribution can be found in at least three distinct places:
the important paper of Dunfield--Thurston testing the virtual Haken conjecture
through random models \cite[\S3.14]{DT}; the recent paper of Kravchenko--Mazur--Petrenko 
on generation of algebras by random elements \cite{KMP}; and the   preprint 
of Wang--Stanley on the Smith normal form distribution of random matrices \cite{WS}. 
These papers use notions of random matrices that differ from the one induced by the Gromov model, but we will explain some of the distinctions below.
By contrast, there are many other ways that random abelian groups arise in mathematics:
as class groups of imaginary quadratic fields, for instance, or 
as cokernels of graph Laplacians for random graphs (also
known as {\em sandpile groups}).  For a discussion
of heuristics for these various distributions 
and a useful survey of some of the random abelian group literature, see \cite{MMW} and its references.

In this paper we initiate a study of random nilpotent
groups by beginning with the {\em free nilpotent group} $N_{s,m}$ of step $s$ and rank $m$
and adding random relators as above.
Note that all nilpotent groups occur as quotients of appropriate $N_{s,m}$, just as all abelian groups
are quotients of some $\Z^m$ and all groups are quotients of some $F_m$ (here and throughout,
 groups are taken to be finitely generated).
By construction, these free nilpotent groups can be thought of as ``nilpotentizations" of Gromov random groups;
their abelianizations  will agree with those described in the last paragraph
(cokernels of random matrices), 
but they have a nontrivial lower central series 
and therefore retain more information about the original random groups.

Below, we begin to study the typical properties of random nilpotent groups.  
For instance, one would expect that the threshold for trivialization occurs
with far fewer relators than for free groups, and also that nontrivial abelian quotients should occur
with positive probability at some range of relator growth.  

The results of this paper are summarized as follows:
\begin{itemize}
\item In the remainder of this section, we establish a sequence
of group theory and linear algebra lemmas for the following parts.
\item In \S\ref{sec-prob}, 
 the properties of $\Z^m$ random walk and its non-backtracking variant are described in order to deduce arithmetic statistics of Mal'cev coordinates.
\item We survey the existing results from which ranks of random abelian groups can be calculated;
 a theorem of Magnus guarantees that the rank of a nilpotent group equals the rank of its abelianization.  (\S\ref{sec-prelim})
\item We give a complete description of one-relator quotients of the Heisenberg group, and 
compute the orders of finite quotients with any number of relators. (\S\ref{sec-heis})
\item Using a Freiheitssatz for nilpotent groups, we study the consequences of rank drop,
and conclude that abelian groups  occur with probability zero for $|R|\le m-2$, while they have positive probability
for larger numbers of relators.
Adding relators in a stochastic process drops the rank by at most one per new relator,
with statistics for successive rank drop given by number-theoretic properties of the 
Mal'cev coordinates.
  (\S\ref{sec-rankdrop})
\item We give a self-contained proof that a random nilpotent group is \aas trivial exactly if $|R|$ is unbounded 
as a function of $\ell$.  We show how information about the nilpotent quotient lifts to 
information about the LCS of a standard (Gromov) random group and observe that standard
 random groups are {\em perfect} under the same conditions.  (\S\ref{sec-perfect})
\item Finally, the last section 
 records experimental data gathered in Sage for random 
quotients of the Heisenberg group, showing in particular the variety of non-isomorphic
groups visible in this model of random nilpotent groups and indicating some of 
their group-theoretic properties. (\S\ref{sec-experiments})
\end{itemize}

\subsection{Nilpotent groups and Mal'cev coordinates}

Nilpotent groups are those for which nested commutators become trivial after a certain uniform depth.
We will adopt the commutator convention that  $[a,b]=aba^{-1}b^{-1}$  and 
define nested commutators on the left by $[a,b,c]=[[a,b],c]$,  $[a,b,c,d]=[[[a,b],c],d]$, 
and so on.  Within a group we will write $[H,K]$ for the subgroup generated by all commutators $[h,k]$ 
with $h$ ranging over $H\le G$ and $k$ ranging over $K\le G$, so that in particular
$[G,G]$ is the usual commutator
subgroup of $G$.  
A group is \emph{$s$-step nilpotent} if all commutators with $s+1$ arguments 
are trivial, but not all those with $s$ arguments are.  (The step of nilpotency is also known as the {\em class}
of nilpotency.)
With this convention, a group is abelian if and only
if it is one-step nilpotent.  References for the basic theory of nilpotent groups are 
\cite[Ch 9]{sims}, \cite[Ch 10-12]{DK}.
  
In the free group $F_m$ of rank $m$, 
let  
$$T_{j,m}=\left\{ \left[a_{i_1},\ldots,a_{i_j}\right] : 
1\le i_1,\ldots,i_j \le m\right\}$$
be the set of all nested commutators with $j$ arguments ranging over the generators.
Then the  {\em free $s$-step rank-$m$ nilpotent group} is 
$$N_{s,m}=\groupmod{F_m}{\nc {T_{s+1,m}}}
=\pres{a_1, \ldots, a_m}{[a_{i_1}, \ldots a_{i_{s+1}}] ~\hbox{for all}~ i_j},$$
where $\nc R$ denotes the normal closure of a set $R$ when its ambient group
is understood.
Just as all finitely-generated groups are quotients
of (finite-rank) free groups, all finitely-generated
nilpotent groups are quotients of free nilpotent groups.  
Note that the standard Heisenberg group $H(\Z)=\pres{a,b}{[a,b,a],[a,b,b]}$ is realized as $N_{2,2}$.  
In the Heisenberg group, we will use the notation $c=[a,b]$, so that the center is $\langle c\rangle$.

The {\em lower central series} (LCS)
for a $s$-step nilpotent group $G$ is 
a sequence of subgroups inductively defined by 
$G_{k+1}=[G_k, G]$ which form a subnormal series 
$$\{1\}=G_{s+1} \lhd \ldots \lhd G_3 \lhd G_2 \lhd G_1=G.$$
(The indexing is set up so that $[G_i,G_j]\subset G_{i+j}$.)
For finitely generated nilpotent groups, this can always be refined to a
 \emph{polycyclic series} 
 $$\{1\}= CG_{n+1} \lhd CG_n \lhd \ldots \lhd CG_2 \lhd CG_1=G$$ 
 where each $CG_i/CG_{i+1}$ is cyclic, so either $\Z$ or $\Z/n_i\Z$. 
 The number of $\Z$ quotients in any
 polycyclic series for $G$ is called the {\em Hirsch length} of $G$.
From a polycyclic series we can form a generating set which supports a useful normal form for $G$.
Make a choice of $u_i $ in each $CG_i$ so that $u_iCG_{i+1}$ generates 
$CG_i/CG_{i+1}$.  An inductive argument 
shows that the set $\{u_1, \ldots, u_n\}$ generates $G$.  
We call such a choice a \emph{Mal'cev basis} for $G$, and we 
 filter it  as $\MB_1\sqcup \dots \sqcup \MB_s$,
with $\MB_j$ consisting of basis elements belonging to $G_j\setminus G_{j+1}$.
Now if $u_i\in \MB_j$, let $\tau_i$ 
be the smallest value such that $u_i^{\tau_i}\in \MB_{j+1}$,
putting $\tau_i=\infty$ if no such power exists.  Then the 
Mal'cev normal form in $G$ is as  follows:
every element $g \in G$ has a unique expression as $g=u_1^{t_1} \cdots u_n^{t_n}$,
with integer exponents and  $0\le t_i\le \tau_i$ if $\tau_i<\infty$.
Then the tuple of exponents
$(t_1,\dots,t_n)$ gives a coordinate system 
on the group, called {\em Mal'cev coordinates}.
We recall that $\MB_j\cup\dots\cup\MB_s$ generates $G_j$ for each $j$ and that
(by definition of $s$) the elements of $\MB_s$ are central.

We will  denote a Mal'cev basis for free nilpotent groups 
$N_{s,m}$ as follows: let $\MB_1=\{a_1,\dots,a_m\}$ be the basic 
generators,
let $\MB_2=\{b_{ij}:=[a_i, a_j] : i<j\}$ be the basic commutators, 
and take each $\MB_j$  as a subset of $T_{j,m}$ consisting of independent 
 commutators from $[\MB_{j-1},\MB_1]$. 
We note that $|\MB_2|={m\choose 2}$, and more generally the orders
are given by the {\em necklace polynomials}
$$|\MB_j|=\frac 1j \sum_{d|j} \mu(d)m^{j/d},$$ where $\mu$ is the M\"obius function
(see \cite[Thm 11.2.2]{Hall}).

For example, the Heisenberg group $H(\Z)=N_{2,2}$ has the lower central series
$\{1\} \lhd \Z \lhd H(\Z)$ with quotients $\Z$ and $\Z^2$, so its Hirsch length is $3$.
$H(\Z)$ admits the Mal'cev basis $a,b,c$ (with $a=a_1$, $b=a_2$, and $c$ equal to 
their commutator), which 
supports a normal form $g=a^A b^B c^C$.  The Mal'cev coordinates of a group element are 
the triple $(A,B,C)\in \Z^3$.

\subsection{Group theory and linear algebra lemmas}

In the free group 
$F_m=\langle a_1,\ldots,a_m \rangle$, for any
freely reduced $g\in F_m$,
we define $A_i(g)$, called the {\em weight} of generator $a_i$ 
in the word $g$, to be the exponent sum of $a_i$
in $g$.
Note that weights $A_1,\ldots,A_m$ are  well defined in the 
same way for the free nilpotent group $N_{s,m}$ for any $s$.  
We will let $\ab$ be the abelianization
map of a group, so that 
$\ab(F_m)\cong \ab(N_{s,m})\cong\Z^m$.  Under this isomorphism, we can identify $\ab(g)$ with the vector 
$\A(g):=(A_1(g),\ldots,A_m(g))\in \Z^m$.  If we have an automorphism $\phi$ on $N_{s,m}$, we write $\phi^{\ab}$ for the induced map on $\Z^m$, which by construction satisfies $\ab\circ\phi=\phi^{\ab}\circ\ab$.
Note that $\A(g)$ is also the $\MB_1$ part of the Mal'cev coordinates for $g$, and 
we can similarly define a $b$-weight vector $\B(g)$ to be the $\MB_2$ part, recording the 
exponents of the $b_{ij}$ in the normal form.

To fix terminology:  the {\em rank} of any finitely-generated group
will be the minimum size of any generating set.  Note this is different from the {\em
dimension} of an abelian group, which we define by
$\dim(\Z^d \times G_0)=d$ for any finite group $G_0$.  
(With this terminology, the Hirsch length of a nilpotent group $G$ is the 
sum of the dimensions of its LCS quotients.)
In any finitely-generated group, 
we say an element is
{\em primitive} if it belongs to some basis 
(i.e., a generating set of minimum size).   
For a vector $w = (w_1,\ldots,w_m) \in \Z^m$, we will
write $\gcd(w)$ to denote the gcd of the entries.
So a vector $w\in\Z^m$ is primitive 
iff
$\gcd(w)=1$.  In this case  we will say that the tuple
$(w_1,\ldots,w_m)$ has the {\em relatively prime property} or is RP.
As we will see below, an element $g\in N_{s,m}$ is primitive in that nilpotent group
 if and only if 
its abelianization is primitive in $\Z^m$, i.e., if $\A(g)$ is RP. 
In free groups, there  {\em  exists} a primitive element with the same abelianization as $g$
iff $\A(g)$ is RP.

The latter follows from a classic theorem of Nielsen \cite{N}.

\begin{theorem}[Nielsen primitivity theorem]\label{nielsen}
For every relatively prime pair of integers $(i,j)$, 
there is a unique
conjugacy class $[g]$ in the free group 
$F_2=\langle a,b\rangle$ for which 
$A(g)=i$, $B(g)=j$, and $g$ is primitive.
\end{theorem}

\begin{corollary}[Primitivity criterion in free groups]\label{RP-prim}
There exists a
primitive element $g\in F_m$ with $A_i(g)=w_i$ for 
$i=1,\ldots,m$
if and only if $\gcd(w_1,\ldots,w_m)=1$.
\end{corollary}

\begin{proof}  Let $w=(w_1,\ldots,w_m)$.
If $\gcd(w)\neq 1$, then
the image of any $g$ with those weights would not be 
primitive in the
abelianization $\Z^m$, so no such $g$ is primitive in $F_m$.

For the other direction we use induction on $m$, 
with the base case $m = 2$ established by Nielsen.
Suppose there exists a primitive element of $F_{m-1}$ with 
given weights $w_1,\ldots,w_{m-1}$.
For $\delta = \gcd(w_1,\ldots,w_{m-1})$,   
we have $\gcd( \delta, w_m ) = 1$.
Let $\overline{w}
= (\frac{w_1}{\delta}, \cdots, \frac{w_{m-1}}{\delta} )$.
By the inductive hypothesis, there exists an element 
$\overline{g} \in F_{m-1}$ 
such that the weights of $\overline{g}$ are $\overline{w}$, 
and $\overline{g}$
can be extended to a basis 
$\{ \overline{g}, {h_2}, \ldots, {h_{m-1}} \}$ 
of $F_{m-1}$.  
Consider the free group 
$\langle \overline{g}, a_m \rangle \cong F_2$.
Since $\gcd( \delta, w_m ) = 1$, there exist
$\hat{g}, \hat{h}$ that generate this free group such that 
$\hat g$ has
weights
$A_{\overline{g}}(\hat{g}) = \delta$ and
$A_m (\hat{g}) = w_m$ by Nielsen.
Consequently, $A_i(\hat{g}) = w_i$.
Then  $\langle \hat{g},\hat{h}, {h_2}, \cdots,
{h_{m-1}}\rangle
=\langle \overline{g},h_2,\ldots,h_{m-1},a_m\rangle=F_m$,
which shows that 
$\hat{g}$ is primitive, as desired.
\end{proof}

The criterion for primitivity in free nilpotent groups easily follows from a powerful
theorem due to Magnus \cite[Lem 5.9]{MKS}.

\begin{theorem}[Magnus lifting theorem] \label{magnus}
If $G$ is nilpotent and $S\subset G$ is any set of elements 
such that
$\ab(S)$ generates $\ab(G)$, then $S$ generates $G$.
\end{theorem}

Note that this implies that if $G$ is nilpotent of rank $m$, 
then $G/ \nc{g}$ has rank at least $m-1$, because we can 
drop at most one dimension in the abelianization.

\begin{corollary}[Primitivity criterion in free nilpotent groups]\label{prim-nilp}
An element $g\in N_{s,m}$ is primitive if and only if $\A(g)$ 
is primitive in $\Z^m$.
\end{corollary}

Now we establish a sequence of lemmas for working with rank and primitivity.
Recall that $a,b$ are the basic generators of the Heisenberg group $H(\Z)$ and 
that $c=[a,b]$ is the central letter.

\begin{lemma}[Heisenberg basis change]\label{basis-change-warmup}
For any integers $i,j$,
there is an automorphism $\phi$
of $H(\Z)=N_{2,2}$ such that $\phi(a^i b^j c^k)=b^d c^m$, where 
$d=\gcd(i,j)$ and $m=\frac{ij}{2d}(d-1) +k$.

In particular, if $i,j$ are relatively prime, then there is
an automorphism $\phi$ of $H(\Z)$ such that $\phi(a^ib^j)=b$.
\end{lemma}

\begin{proof}
Suppose $ri+sj=d=\gcd(i,j)$ for integers $r,s$ and consider
$\hat a=a^{s}b^{-r}$, $\hat b=a^{i/d} b^{j/d}$.
We compute
$$[a^{s}b^{-r},a^{i/d} b^{j/d} ]=[a^s,b^{j/d}]\cdot [b^{-r},a^{i/d}]
=c^{(ri+sj)/d}=c.$$
If we set $\hat c=c$, we  have $[\hat a,\hat b]
=\hat c$ and
$[\hat c,\hat a]=[\hat c,\hat b]=1$, so $\langle \hat a,\hat b\rangle$
presents a quotient of the Heisenberg group.
We need to check that it is the full group.
Consider $h=(\hat a)^{-i/d} (\hat b)^s$.
Writing $h$ in terms of $a,b,c$, the $a$-weight of $h$ is $0$ and the
$b$-weight is $(ri+sj)/d=1$, so $h=bc^t$ for some $t$.  But then
$b=(\hat a)^{-i/d}(\hat b)^s(\hat c)^{-t}$ and similarly 
$a=(\hat a)^{j/d}(\hat b)^r(\hat c)^{-t'}$ for some $t'$,
so all of $a,b,c$ can be expressed in terms of $\hat a,\hat b,\hat c$.

Finally, $$(\hat b)^d =(a^{i/d}b^{j/d})^d=a^ib^j c^{-{d \choose 2}\frac{ij}{d^2}},$$
which gives the desired expression $a^i b^j c^k = (\hat b)^d (\hat c)^m$ from   above.
\end{proof}

\begin{prop}[General basis change]\label{basis-change}
Let $\delta=\gcd(A_1(g),\ldots,A_m(g))$ for any $g\in H=N_{s,m}$.
Then there is an automorphism $\phi$
of $H$ such that $\phi(g)=a_m^\delta \sdot h$ for some $h\in H_2$.
\end{prop}

\begin{proof}
Let $w_i=A_i(g)$ for $i=1,\ldots,m$ and let $r_i=w_i/\delta$, so that
$\gcd(r_1, \ldots, r_m)=1$.  By Corollary~\ref{RP-prim},
there exists a primitive element $x \in F_m$ with weights $r_i$.
Let $\phi$ be a change of basis automorphism of $F_m$ such that $\phi(x)=a_m$.
This induces an automorphism of $H$, which we will also call $\phi$.

By construction, $x^{\delta}$ and $g$ have weight $w$.  
Since $\ab(x^{\delta})=\ab(g)=w$, we must have $\phi^{\ab}(w) = \ab(\phi(x^{\delta}))= \ab(\phi(g))$.  Therefore $\phi(x^{\delta})$ and $\phi(g)$ have the same weights.

Then $\ab(\phi (g)) = \ab(\phi(x^\delta))=\ab(\phi(x)^\delta)=\ab(a_m^\delta)$, 
so $\phi(g)$ and $a_m^{\delta}$ only differ by commutators, 
i.e., $\phi(g)=a_m^\delta \cdot h$ for some $h \in H_2$.
\end{proof}

\begin{remark}\label{fgabelclassification} Given an abelian group $G=\Z^m/\langle R\rangle$, 
the classification of finitely-generated abelian groups 
provides that there are non-negative integers
$d_1, \ldots, d_m$  with $d_m| d_{m-1}| \ldots | d_1$ such that
$G \cong  \bigoplus_{i=1}^m \Z / d_i \Z$.
If $G$ has dimension $q$ and rank $r$, then
$d_1=\dots=d_q=0$, and $d_{r+1}=\dots=d_m=1$, so that 
$$G\cong \Z^q \times \left( \Z/d_{q+1}\Z \times \dots \times \Z/d_r\Z\right).$$
Now consider a projection map $f:\Z^m \to \Z^m/K \cong \bigoplus_{i=1}^m \Z / d_i \Z$.
We can choose a basis $e_1, \ldots, e_m$ of $\Z^m$
so that
$$K =\spn\{ d_1 e_1, \ldots, d_m e_m \}
\cong \bigoplus_{i=1}^m d_i \Z.$$
Then since every element in $K$ is a linear combination of $\{ d_1 e_1, \ldots, d_m e_m \}$ and $d_m| d_{m-1}| \ldots | d_1$, we have that 
$d_m$ divides all the coordinates of all the elements in $K$. Also $d_me_m\in K$ with $e_m$ being primitive.
\end{remark}

\begin{lemma}[Criterion for existence of primitive vector]\label{gcd-rm}
Consider a set of $r$ vectors in $\Z^m$, and let $d$ be the gcd of the $rm$ coordinate
entries.  Then there exists a vector in the span such that the gcd of its entries is 
$d$, and this is minimal among all vectors in the span.

In particular, a set of $r$ vectors in $\Z^m$ has a primitive vector in its span if and 
only if the gcd of the $rm$ coordinate entries is $1$.
\end{lemma}

\begin{proof}
With $d$ as above, let $K$ be the $\Z$-span of the vectors and let
$$\gamma:=\inf_{w\in K} \gcd(w).$$
One direction is clear:  every vector in the span has every coordinate divisible by $d$, so $\gamma \ge d$.  On the other hand $d_me_m\in K$ and $\gcd(d_me_m)=d_m$
because $e_m$ is primitive.  But $d_m$  is a common divisor of all $rm$ 
coordinates, and $d$ is the greatest such, so $d_m\le d$ and thus $\gamma\le d$.
\end{proof}

\begin{lemma}[Killing a primitive element]\label{prim-kill}
Let $H=N_{s,m}$ and let $K$ be a normal subgroup of $H$. If $\rank(H/K) <m$ then $K$ contains a
primitive element.
\end{lemma}

\begin{proof}
Since $\rank(H/K) < m$, we also have $\rank(\ab(H/K)) < m$. 
Writing $\ab(H/K) \cong \bigoplus_{i=1}^m \Z / d_i \Z$ as above, we have $d_m=1$. 
By the previous lemma there is a primitive element in the kernel
of the projection $\ab(H)\to \ab(H/K)$, 
and any  preimage in $K$ is still primitive (see Cor~\ref{prim-nilp}).
\end{proof}

\begin{lemma}[Linear algebra lemma]\label{turbo-prop}
Suppose $u_1,\dots,u_n \in\Z^m$ and suppose there exists a primitive vector
$v$ in their span.  Then there exist $v_2,\dots,v_n$ such that
$\spn(v,v_2,\dots,v_n)=\spn(u_1,\dots,u_n)$.
\end{lemma}

\begin{proof}
Since $v\in\spn(u_1,\dots,u_n)$, we can write $v=\alpha_1u_1+\dots+\alpha_n u_n$.
Let $x\in\Z^n$ be the vector with coordinates $\alpha_i$.
Because $\gcd(v)=1$, we have $\gcd(\alpha_i)=1$, so $x$ is primitive.
Thus, we can complete $x$ to a basis of $\Z^n$, say $\{x,x_2,\dots,x_n\}$.
Then take
$\left(\begin{array}{c}-v-\\ -v_2-\\ \vdots \\ -v_n-\end{array}\right) = \left(\begin{array}{c}-x-\\ -x_2-\\ \vdots \\ -x_n-\end{array}\right)\cdot
\left(\begin{array}{c}-u_1-\\ -u_2-\\ \vdots \\ -u_n-\end{array}\right)$.
Since
$\left(\begin{array}{c}-x-\\ -x_2-\\ \vdots \\ -x_n-\end{array}\right)\in SL_n(\Z)$, it represents a change of basis matrix, so
we have $\spn(v,v_2,\cdots,v_n)=\spn(u_1,\cdots,u_n)$, as needed.
\end{proof}

\begin{lemma}[String arithmetic]\label{string-arith}
Fix a free group $F=F_m$ on $m$ generators and let
$R,S$ be arbitrary subsets, with normal closures
$\nc R, \nc S$.  Let $\phi:F\to F/\nc R$ and
$\psi: F\to F/\nc S$ be the quotient homomorphisms.
Then there exist canonical isomorphisms
$$ (F/\nc R) \big\slash \nc{\phi(S)}
\cong F \big\slash \nc{R\cup S}
\cong (F/\nc S) \big\slash \nc{\psi(R)}$$
that are compatible with the underlying presentation
(i.e., the projections from $F$ commute with
these isomorphisms).
\end{lemma}

\begin{proof}
We will abuse notation by writing strings from $F$
and interpreting them in the various quotients we
are considering.  Then if
$G=\langle F\mid T\rangle\cong F/\nc T$ is a quotient of $F$
and $U$ is a subset of $F$, we can write
$\langle G \mid U \rangle$ to mean
$F\big\slash \nc{T\cup U}$ and can equally well write
$\langle F \mid T,U\rangle$.
Then the isomorphisms
we need just record the fact that
$$\pres F{R,S}
=\pres{F/\nc R}S
=\pres{F/\nc S}R.\qedhere
$$
\end{proof}
\noindent Because of this standard abuse of notation where we will variously
interpret a string in $\{a_1,\dots,a_m\}^\pm$ as belonging to $F_m$,
$N_{s,m}$, or some other quotient group, we will use the symbol
$=_G$ to denote equality in the group $G$ when trying to emphasize
the appropriate ambient group.

\section{Random walk and arithmetic uniformity}\label{sec-prob}

In this section we record some properties of the simple nearest-neighbor random walk
(SRW) and the non-backtracking random walk (NBSRW) on the integer lattice
$\Z^m$, then deduce consequences for the distribution of Mal'cev coordinates
for random relators in free nilpotent groups. (See the Appendix for additional details \cite{appendix}.)
For the standard basis $\{e_i\}$ of $\Z^m$,
SRW is defined by giving the steps $\pm e_i$  equal probability $1/2m$, and 
NBSRW is similarly defined but with the added condition that the step 
$\pm e_i$ cannot be immediately followed by the step $\mp e_i$ (that is,
a step can't undo the immediately previous step; equivalently, the position after $k$
steps cannot equal the position after $k+2$ steps).  
A random string $w_\ell$ of $\ell$ letters from $\{a_1,\dots,a_m\}^\pm$
has the form  $w_\ell = \alpha_1 \alpha_2 \cdots \alpha_\ell$, where the $\alpha_i$
are i.i.d. random variables which equal each basic generator or its inverse with 
equal probability $1/2m$.  The  abelianization 
$X_\ell=\A(w_\ell)$ is a $\Z^m$-valued random variable corresponding
to $\ell$-step SRW.  A random freely reduced string does not have an expression
as a product of variables identically distributed under the same law, but if $v_\ell$ is such 
a string, its weight vector $Y_\ell=\A(v_\ell)$ is another $\Z^m$-valued random 
variable, this time corresponding to NBSRW.

It is well known that the distribution of endpoints for a simple random walk 
in $\Z^m$ converges to a multivariate Gaussian:  if $X_\ell$ is
again  the random variable recording the endpoint
after $\ell$ steps of simple random walk on $\Z^m$, and  
$\delta_t$ is the dilation in $\R^m$ sending $v\mapsto tv$, 
we have the central limit theorem:
$$\delta_{\frac{1}{\sqrt\ell}} X_\ell \longrightarrow
  \mathcal{N}(\mathbf{0},\tfrac 1m I).$$
This convergence notation for a vector-valued random variable $V_\ell$ and
a multivariate normal 
 $\mathcal{N}(\mu,\Sigma)$ means that 
$V_\ell$ converges in distribution to 
$A W +\mu$, where the vector $\mu$ is the mean,
 $\Sigma=AA^T$ is the covariance matrix, and $W$ is a vector-valued random variable
  with i.i.d.\ entries drawn 
 from a standard (univariate) Gaussian distribution $\mathcal{N}(0,1)$.
 In other words, this central limit theorem tells us that 
 the individual entries of $X_\ell$ are asymptotically independent, Gaussian
 random variables with mean zero and expected magnitude $\sqrt{\ell}/m$.
This is a special case of a much more general result of Wehn for Lie groups 
and can be found for instance in \cite[Thm 1.3]{B-survey}.  
Fitzner and van der Hofstad derived a corresponding central limit theorem
for NBSRW in \cite{FH}.  Letting $Y_\ell$ be the $\Z^m$-valued random variable
for $\ell$-step NBSRW as before, they find that for $m\ge 2$,
$$\delta_{\frac{1}{\sqrt\ell}} Y_\ell \longrightarrow \mathcal{N}(\mathbf{0},\tfrac 1{m-1} I).$$
Note that the difference between the two statements records something intuitive:
the non-backtracking walk still has mean zero, but the rule causes the expected
size of the coordinates to be slightly higher than in the simple case; also, it blows
up (as it should) in the case $m=1$.

The setting of nilpotent groups is also well studied.  
To state the central limit theorem for free nilpotent
groups, we take 
$\delta_t$ to be the similarity which scales each coordinate from $\MB_j$
by $t^j$, so that for instance in the Heisenberg group, 
$\delta_t(x,y,z)= (tx,ty,t^2z)$.

\begin{proposition}[Distribution of Mal'cev coordinates]\label{malcevdistribution}
Suppose $\NB_\ell$ is an $N_{s,m}$-valued random variable 
chosen by  non-backtracking simple random walk (NBSRW) on 
$\{a_1,\ldots,a_m\}^\pm$ for $\ell$ steps.  
Then the distribution on the Mal'cev coordinates  is
asymptotically normal:
$$ 
\delta_{\frac{1}{\sqrt\ell}} \NB_\ell \sim \mathcal{N}(\mathbf{0},\Sigma).
$$
\end{proposition}

For SRW, this is called a ``simple corollary" of Wehn's theorem in  \cite[Thm 3.11]{B-survey}), 
where the only hypotheses are that the steps of the random walk are i.i.d.\ under a probability measure on $N_{s,m}$ that is centered, with finite second moment (in this case, 
the measure has finite support, so all moments are finite).
Each Mal'cev coordinate 
is given by a polynomial formula in the $a$-weights
of the variables $\alpha_i$ (the polynomial for an $MB_j$ coordinate has degree $j$---for instance in $H(\Z)=N_{2,2}$ the coordinate $C$ is a quadratic in $\A(\alpha_i)$).  The number of summands in the polynomial gets large as $\ell\to\infty$.
Switching to NBSRW, it is still the case that $\NB_\ell$ is a product of group elements
whose $a$-weight vectors are 
independent and normally distributed, so their images under the same polynomials
will be normally distributed as well, with only the covariance differing from the SRW case.
We sketch a simple and self-contained argument for this in 
the $N_{2,2}$ non-backtracking case---that the third
Mal'cev coordinate in $H(\Z)$ is normally distributed---
which we note is easily generalizable to the other $N_{s,m}$ with (only) considerable notational suffering.
Without loss of generality, the sample path of the random walk is 
$$g=a^{i_1}b^{j_1} a^{i_2}b^{j_2} \dots a^{i_r}b^{j_r}$$
for some integers $i_s,j_t$ summing to $\ell$ or $\ell-1$,
with all but possibly $i_1$ and $j_r$ nonzero.
After a certain number of steps, suppose the last letter so far was $a$.  Then the next letter
is either $a$, $b$, or $b^{-1}$ with equal probability, so there is a $1/3$ chance of repeating
the same letter and a $2/3$ chance of switching.
This means that the $i_s$ and $j_t$ 
are (asymptotically independent) run-lengths of heads for a biased coin (Bernoulli
trial) which lands heads 
with probability $1/3$.  
On the other hand, $r$ is half the number of tails flipped by that coin in $\ell$ (or $\ell-1$) trials.  
In Mal'cev normal form, 
$$g=a^{\sum i_s}b^{\sum j_t}c^{\sum_{t<s}i_s j_t}.$$
Thus the exponent of $c$ is obtained by adding products of run-lengths together ${r \choose 2}$ 
times, and general central limit theorems ensure that 
adding many independent and identically distributed (i.i.d.) random variables together 
tends to a normal distribution.

Our distribution statement has
 a particularly nice formulation in this Heisenberg case,
where the third Mal'cev coordinate records the {\em signed area} enclosed between the 
$x$-axis and the path traced out by a word in $\{a,b\}^\pm$.  For instance, in the figure below, $baba$ encloses area $-3$, which equals the $c$ exponent in the normal form.

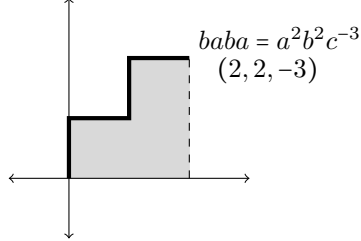
\centerline{
\begin{tikzpicture}[scale=.8]
\filldraw [gray!30]  (0,0)--(0,1)--(1,1)--(1,2)--(2,2)--(2,0)--cycle;
\draw [ultra thick] (0,0)--(0,1)--(1,1)--(1,2)--(2,2);
\draw [dashed] (2,0)--(2,2);
\draw [<->] (0,-1)--(0,3);
\draw [<->] (-1,0)--(3,0);
\node at (2,2) [above right] {$baba=a^2b^2c^{-3}$};
\node at (3.3,1.8) {$(2,2,-3)$};
\end{tikzpicture}
}

\begin{corollary}[Area interpretation for Heisenberg case] 
For the simple random walk on the plane, the signed area enclosed by the path 
is a normally distributed random variable.
\end{corollary}

Next, we want to describe the effect of a group automorphism on the 
distribution of coordinates.  Then we conclude this section by considering
the distribution of coordinates in various $\Z/p\Z$.

\begin{corollary}[Distributions induced by automorphisms]\label{chg-var-dist}
If $\phi$ is an automorphism of $N_{s,m}$ and $g$ is a random freely
reduced word of length $\ell$ in $\{a_1,\dots,a_m\}^\pm$, then 
the Mal'cev coordinates of $\ab(\phi(g))$ are also normally distributed.
\end{corollary}

\begin{proof}The automorphism
$\phi$ induces a change of basis on the copy of $\Z^m$ in the $\MB_1$
coordinates, which is given by left-multiplication by a matrix $B\in SL_m(\Z)$.
Then $\phi_*(Y_\ell) \to \mathcal{N}({\bf 0},B\Sigma B^T)$. \end{proof}

Note that normality of the $\MB_j$ coordinates follows as well, as before:
they are still described by sums of statistics coming from asymptotically independent 
Bernoulli trials, and only the coin bias has changed.  

Relative primality of $\MB_1$ coefficients turns out to be the key to studying 
the rank of quotient groups, so we will need some arithmetic lemmas.

\begin{lemma}[Arithmetic uniformity]\label{arith-unif}
Let $A_{\ell,i}$ be the $\Z$-valued random variable given by the $a_i$-weight 
of a random word of length $\ell$ in $\{a_1,\dots,a_m\}^\pm$,
for $1\le i\le m$.
Let $\hat A_{\ell,i}$ equal $A_{\ell,i}$ with probability $\frac 12$ and 
$A_{\ell-1,i}$ with probability $\frac 12$.
Then $\forall \epsilon>0 \quad \exists c_1,c_2>0$ s.t. 
$$  \forall n\le\ell^{\frac 12 -\epsilon}, \ \forall k, 
\qquad \Pr\Bigl(\hat A_{\ell,i} \equiv k \mod n\Bigr) < 
\frac 1n +  c_1 e^{ -c_2\ell^{2\epsilon}}.$$
More generally
$\exists c_1,c_2>0$ s.t. for any $s\le m$   and distinct
$i_1,\dots,i_s$,  
$$ \forall n\le\ell^{\frac 12 -\epsilon}, \ \forall k_1,\dots,k_s,
\qquad
\Pr\Bigl(\hat A_{\ell,i_1}\equiv k_1, \cdots, \hat A_{\ell,i_s}
\equiv k_s \mod n\Bigr)<
\frac{1}{n^s} +c_1e^{-c_2\ell^{2\epsilon}}.$$
\end{lemma}
In other words, the $\Z/n\Z$-valued random variables induced by the coordinate projections
from random walk on the Mal'cev generators 
$\MB_1$ approach independent uniform distributions.  

\begin{proof}  Depending on whether the random word is chosen as a random string
or a random freely reduced string, 
 $A_{\ell,i}$ is the $i$th coordinate projection of 
 either the SRW $X_\ell$ or the NBSRW $Y_\ell$, and 
$\hat A_{\ell,i}$ is the parity-corrected version.
We first consider the residues mod $n$ for simple random walk $X_\ell$ on $\Z^m$
by studying its position on the discrete torus $(\Z/n\Z)^m$.  
We only need to consider the statement for $s$ coordinates in the  
case $s=m$, since the results for $s<m$ 
can be derived from this by summing:  for instance, the positions
satisfying $\pi_i(X_\ell)\equiv k_i$ for $i=1,2$ are represented by $n^{m-2}$ positions
in the torus, so the bound can be added polynomially many times to get the right main 
term, at the cost of 
slightly enlarging the constant $c_2$.

A theorem from Saloff-Coste \cite[Theorem 7.8]{SC} controls
the distance from a lazy symmetric generating random walk 
to the uniform distribution on any family of 
finite Cayley graphs which satisfies a uniform doubling bound on volume growth.  
(In our case, the growth $\#B_r$ is bounded on $(\Z/n\Z)^m$ 
by $(2r+1)^m$, independent of $n$, and the graphs satisfy the doubling hypothesis.)
First we will explain how this theorem provides the needed bound, then
we will explain how to modify our random walk to satisfy the
theorem's hypotheses.

In our notation, the theorem says that 
$$\Pr\Bigl(\pi_{1}(X_\ell)\equiv k_1, \cdots, \pi_m(X_\ell)\equiv k_m \mod n\Bigr) < 
\frac 1{n^m} +  \frac{c_1\sdot n^m}{\ell^{m/2}} e^{-c_2 \ell / (mn)^2},$$
for the following reasons:
 the $L^2$ distance upper-bounds the difference in probabilities at any 
single point, and the diameter of $(\Z/n\Z)^m$ is less than $mn$.
Since $n<\ell^{\frac 12-\epsilon}$, we have $n^2<\ell^{1-2\epsilon}$, and 
by enlarging the constants we obtain 
$$\Pr\Bigl(\pi_{1}(X_\ell)\equiv k_1, \cdots, \pi_m(X_\ell)\equiv k_m \mod n\Bigr) < 
\frac 1{n^m} +  c_1 e^{-c_2 \ell^{2\epsilon}},$$
as desired.

In order to use this theorem on our (non-lazy) walk,  we apply the following technique:
we replace the simple random walk $P$ with the two-step walk $P*P$ which is lazy
and symmetric.  If $n$ is odd, the support of $P*P$ is a generating set, and we 
can proceed.
If $n$ is even, $P*P$ is supported on the sublattice of torus points where the sum of the 
coordinates is even, which does not generate.
But in that case the random variable $\hat A_\ell$ that we are studying (which takes
either $\ell$ or $\ell-1$ steps)
lives on the even or odd sublattice with equal probability; Saloff-Coste's statement
 will ensure equidistribution on the even sublattice, and by symmetry, 
taking one more step will equidistribute on the odds.  
(To be precise, we should use $\ell/2$ rather than $\ell$ on the right-hand
side because of the parity fix, but this gets absorbed in the constants.)

To handle  NBSRW, we can construct a new state space 
whose states correspond to directed edges on the discrete torus; this encodes
the one step of memory required to avoid backtracking. 
This new state space
can itself be rendered as a homogeneous finite graph, and a similar argument
can be applied.
\end{proof}

\begin{corollary}[Uniformity mod $p$]\label{unif-mod-p}
The abelianization of a random freely reduced word in $F_m$ has entries
that are asymptotically uniformly distributed in $\Z/p\Z$ for each prime $p$, 
and the distribution mod $p$ is asymptotically independent of the distribution mod $q$
for any distinct primes $p,q$.
\end{corollary}

\begin{proof}For independence, consider $n=pq$ in the previous Lemma.
Asymptotically uniform implies asymptotically independent.\end{proof}

\begin{corollary}[Probability of primitivity]
\label{rel-prime-coords}
For a random freely reduced word in $F_m$, the probability that it is primitive in abelianization
tends to  $1/\zeta(m)$, where $\zeta$ is the Riemann zeta function.  In particular, for $m=2$,  the probability is $6/\pi^2$.
\end{corollary}

\begin{proof}
Using arithmetic uniformity, one derives a probability expression
that agrees with $1/\zeta(m)$ by Euler's product formula for the zeta function, 
as in \cite{HW}.  For details, see the Appendix \cite{appendix}.
\end{proof}


\begin{remark}[Comparison of random models]\label{model-compare}
As we have seen, abelianizations of Gromov random groups are computed as cokernels of  random matrices $M$
whose columns are given by non-backtracking simple random walk on $\Z^m$.
Most other models in the random abelian groups literature use somewhat different randomization set-ups.
Dunfield and Thurston \cite{DT} use a lazy random walk:  $\ell$ letters are chosen uniformly from the $(2m+1)$ 
possibilities of $a_i^\pm$ and the identity letter, creating a word of length $\le \ell$, whose abelianization becomes 
a column of $M$.  Results by
Kravchenko--Mazur--Petrenko \cite{KMP} and Wang--Stanley \cite{WS} use the standard ``box" model: 
 integer entries are drawn uniformly at random from 
$[-\ell,\ell]$, and asymptotics are calculated as $\ell\to\infty$.  (This is the most classical way to randomize 
integers in number theory; see \cite{HW}.)

However, the main arguments in each of these settings rely on arithmetic uniformity of coordinates mod $p$
to calculate probabilities of relative primality, which is why the Riemann zeta function comes up repeatedly in the calculations.
\end{remark}


\section{Preliminary facts about random nilpotent groups via abelianization} \label{sec-prelim}

In this section we make a few  observations relevant to the model
of random nilpotent groups we study below.  In particular, there has been 
substantial work on quotients of free abelian groups $\Z^m$ by random 
lattices, so it is important to understand the relationship between a random nilpotent group and its abelianization.
Below, and throughout the paper, recall that probabilities are asymptotic as $\ell\to\infty$.

First, we record the simple observation that depth in the LCS is respected by homomorphisms.
 \begin{lemma} \label{homomolcs}
Let $\phi : G \to H$ be a surjective group homomorphism. Then $\phi(G_k)=H_k$ where $G_k$, $H_k$ are the 
level-$k$ subgroups in the respective lower central series.
\end{lemma}

\begin{proof}
Since $\phi$ is a homomorphism,
depth-$k$ commutators are mapped to depth-$k$ commutators, i.e.,
$\phi(G_k) \subseteq H_k$. Let $h\in H_k$.
Without loss of generality we can assume $h$ is a single
nested commutator  $h=[w_1, \dots, w_k]$.
By surjectivity of $\phi$ we can choose lifts
$\overline{w_1}, \ldots, \overline{w_k}$ of
$w_1, \dots, w_k$. We see $[\overline{w_1}, \dots, \overline{w_k}]\in G_k$ and
$\phi(G_k) \supseteq H_k$.
\end{proof}

To begin the consideration of ranks of random nilpotent groups, note that 
the Magnus lifting theorem (Theorem~\ref{magnus}) tells us the rank of $N_{s,m}/\nc R$ equals the rank
of its abelianization $\Z^m/\langle R\rangle$, so we quickly deduce the probability of rank drop.

\begin{prop}[Rank drop]
For a random $r$-relator  nilpotent group $G=N_{s,m}/\nc{g_1,\dots,g_r}$, 
$$\Pr(\rank(G)<m)=\frac{1}{\zeta(rm)}.$$
\end{prop}

\begin{proof} This follows directly from considering the existence of a primitive element in 
$\langle {\ab(R)} \rangle$.  By Lemma~\ref{gcd-rm}, this occurs if and only if the $rm$ entries are
relatively prime, and by arithmetic uniformity (Lemma~\ref{arith-unif}), this is computed
by the Riemann zeta function, as in Corollary~\ref{rel-prime-coords}.
\end{proof}

Next we observe that a nilpotent group is trivial if and only if its abelianization
(i.e., the corresponding $\Z^m$ quotient)
is trivial, and more generally it is finite if and only if the abelianization is finite.  
Equivalence of triviality follows directly from the Magnus lifting theorem (Theorem~\ref{magnus}).  
For the other claim, suppose the abelianization is finite.  Then powers of all the 
images of $a_i$ are trivial in the abelianization, so in the nilpotent group $G$ there 
are finite powers $a_i^{r_i}$ in the commutator subgroup $G_2$.  
A simple inductive argument   shows that every element of $G_j$
has a finite power in $G_{j+1}$; for example, consider $b_{ij}\in G_2$.
Since
 $[a_i^{r_i},a_j]=b_{ij}^{r_i}$ is a commutator of elements from 
$G_2$ and $G_1$, it must be in $G_3$, as claimed.
But then we can see that there
are only finitely many distinct elements in the group by considering the Mal'cev normal form
$$g=u_1^* u_2^* \dots u_r^*$$
and noting that each exponent can take only finitely many values.
Since the rank of a nilpotent group equals that of its abelianization (by Theorem~\ref{magnus}
again), it is also true that a nilpotent group is cyclic if and only if its abelianization
is cyclic.

We  introduce the term {\em balanced} for groups presented with 
the number of relators equal to the number of generators, so that it applies
to models of random groups $F_m/\nc R$, random nilpotent groups $N_{s,m}/\nc R$, 
or random abelian groups $\Z^m/\langle R\rangle$,
 where $|R|=m$, the rank of the seed group.  We will 
correspondingly use the terms {\em nearly-balanced} for $|R|=m-1$, 
and {\em underbalanced} or {\em overbalanced} in the $|R|<m-1$ and $|R|>m$ cases, respectively.

Then it is very easy to see that nearly-balanced (and thus underbalanced)
groups are \aas infinite, while balanced
(and thus overbalanced) groups are  \aas finite because   $m$ 
random integer vectors (in any of the models) 
are $\R$-linearly independent with probability one.
However,  is also easy to see  that if $|R|$ is held constant, no matter how large, 
then there is a nonzero probability that the group is nontrivial (because, for example, all the 
$a$-weights could be even).

To set up the statement of the next lemma, let $Z(m):= \zeta(2)\dots \zeta(m)$ and 
$$P(m):= \prod_{\hbox{\scriptsize primes}~p} \left( 1+ \frac{1/p - 1/p^m}{p-1} \right).$$

As in Remark~\ref{model-compare}, we can quote the distribution
results of \cite{DT},\cite{KMP},\cite{WS} because of the common feature of arithmetic 
uniformity.

\begin{lemma}[Cyclic quotients of abelian groups]\label{lem-cyclic}
The probability that the quotient of $\Z^m$ by $m-1$ random vectors is 
cyclic is $1/Z(m)$.  
With $m$ random vectors, the probability is $P(m)/Z(m)$.  
\end{lemma}

These facts, particularly the first, 
can readily be derived ``by hand," but can also be computed using 
 Dunfield--Thurston \cite{DT} as follows:
their generating functions give expressions for the probability that $i$ random vectors
with $\Z/p\Z$ entries generate a subgroup of rank $j$, and the product over primes 
of the probability that  the $\Z/p\Z$ reduction has rank $\ge m-1$ produces the probability of
a cyclic quotient over $\Z$.

The latter fact appears directly in Wang--Stanley \cite{WS}
as Theorem 4.9(i).
We note that corresponding facts
for higher-rank quotients could also be derived from either of these two papers, but 
the expressions have successively less succinct forms.

\begin{corollary}[Explicit probabilities for cyclic quotients] 
For balanced and nearly-balanced presentations,
the probability that a random abelian group or a random nilpotent group is cyclic is a strictly 
decreasing function of $m$ which converges as $m\to\infty$.
\end{corollary}

In the balanced case, the limiting value is a well-known number-theoretic invariant.
Values are estimated in the table below.

The convergence for both cases
is proved in \cite[Thm 4.9]{WS} as a corollary of the more general statement about the Smith 
normal form of a random not-necessarily-square matrix $M$, which is an expression $A=SMT$ for 
invertible $S,T$ in which $A$ has all zero entries except possibly its diagonal entries $a_{ii}=\alpha_i$.  
These $\alpha_i$ are then the abelian invariants for the quotient of $\Z^m$ by the column span of $M$
(that is, they are the $d_i$ from Remark~\ref{fgabelclassification} but with opposite indexing, $d_i=\alpha_{m+1-i}$).
The rank of the quotient is  the number of these that are not equal to $1$.  

The probabilities of cyclic groups among balanced and nearly-balanced quotients of
free abelian groups and therefore also for random nilpotent groups
are approximated below.  Values in the table are truncated (not rounded) at four digits.

$$\begin{array}{c|ccccccc}
\Pr({\rm cyclic}) & m=2 & m=3 & m=4 & m=10 & m=100 & m=1000 & 
m\to \infty \\
\hline
|R|=m-1 & .6079 & .5057 & .4672 & .4361 & .4357 & .4357 & .4357 \\
|R|=m & .9239 & .8842 & .8651 & .8469 & .8469 & .8469 & .8469 
\end{array}
$$

Computing the probability of a trivial quotient with $r$ relators is equivalent to the 
the probability that $r$ random vectors generate $\Z^m$.  
\begin{lemma}[Explicit probability of trivial quotients]\label{lem-collapse} For $r>m$, 
 $$\Pr\left( \groupmod{\Z^m}{\langle v_1,\dots,v_r\rangle}=0\right) = \frac{1}{\zeta(r-m+1)\cdots \zeta(r)}.$$  
\end{lemma}

This is a rephrasing of  \cite[Cor 3.6]{KMP} and  \cite[Thm 4.8]{WS}.

\begin{remark}
From the description  of Smith normal form, we get a symmetry in $r$ and $m$, namely
$$\Pr\left(   \rank\left(\groupmod{\Z^m}{\langle v_1,\dots,v_r \rangle}\right) = m-k \right)  =
\Pr\left(   \rank\left(\groupmod{\Z^r}{\langle v_1,\dots,v_m \rangle}\right) = r-k \right)  
\quad \forall 1\le k \le \min(r,m)$$
just by the observation that the transpose of the normal form expression has the same invariants.
For example, applying duality to Lemma~\ref{lem-cyclic} and reindexing, we immediately obtain, as in Lemma~\ref{lem-collapse},
$$\Pr\left( \groupmod{\Z^m}{\langle v_1,\dots,v_{m+1}\rangle}=0\right)=\frac{1}{Z(m+1)}=\frac{1}{\zeta(2)\cdots \zeta(m+1)}.$$ 
\end{remark}

\input{rank-drop-viz.tex}

\section{Quotients of the Heisenberg group} \label{sec-heis}

We will classify all $G:=H(\Z)/\nc g$ for single relators $g$, up to isomorphism.  
As above, we write $a,b$ for the generators of $H(\Z)$, and $c=[a,b]$.
With this notation, $H(\Z)$ can be written as a semidirect product $\Z^2 \rtimes \Z$
via $\langle b,c \rangle \rtimes \langle a\rangle$ with the action of $\Z$ on
$\Z^2$ given by $ba=abc^{-1}$, $ca=ac$.

\begin{theorem}[Classification of one-relator Heisenberg quotients]\label{classification}
Suppose $g=a^i b^j c^k\neq 1$.
Let $d=\gcd(i,j)$, let $m=\frac{ij}{2d}(d-1)+k$ as in Lemma~\ref{basis-change-warmup}, and let $D=\gcd(d,m)$.  
Then
$$G:=\groupmod{H(\Z)}{\llangle g \rrangle} \cong
\begin{cases}
\left(\Z\times \Z/k\Z\right) \rtimes \Z, & {\rm if}~ i=j=0;\\
(\Z/\frac{d^2}{D}\Z \ \times \ \Z/D\Z)\rtimes \Z,& \hbox{\rm else,}
\end{cases}
$$
with the convention that  $\Z/0\Z=\Z$ and $\Z/1\Z=\{1\}$.
In particular, $G$ is abelian if and only if $g=c^{\pm 1}$ or $\gcd(i,j)=1$; otherwise, it has step two.
Furthermore, unless $g$ is a power of $c$ (the $i=j=0$ case),
the quotient group is virtually cyclic.  
\end{theorem}

Note that this theorem is exact, not probabilistic.

\begin{remark}[Baumslag-Solitar case]\label{BS-case}
The {\em Baumslag-Solitar groups} are a famous class of groups given by the 
presentations $BS(p,q)=\langle a,b \mid ab^pa^{-1}=b^{q}\rangle$ for various $p,q$.
For the Heisenberg quotients as described above, we will refer to $D=1$ as 
the Baumslag-Solitar case, because in that case $sd-tm=1$ has solutions in $s,t$, and one easily checks that the 
group is presented as $$G=\langle a,b ~\big\mid ~ [a,b]=b^{td}, \ \ b^{d^2}=1\rangle \cong \groupmod{BS(1,1+td)}{\nc{b^{d^2}}},$$
a 1-relator quotient of a solvable Baumslag-Solitar group $BS(1,q)$.
\end{remark}

Examples:
\begin{enumerate}
\item if $g=a$, then $G=\Z$.  
\item if $g=c$, then $G=\Z^2$.  
\item if $g=c^2$, then $G=(\Z\times\Z/2\Z)\rtimes \Z$.
\item if $g=a^{20}b^{28}c^{16}$, we have
$d=4$, $m=226$, $D=2$, so we get
$$G=\left(\groupmod{\Z^2}{\left\langle \vv{4}{226}, \vv{0}{4} \right\rangle} \right)\rtimes \Z 
\cong \left(\groupmod{\Z^2}{\left\langle \vv 42, \vv 04 \right\rangle}\right) \rtimes \Z \cong (\Z/8\Z \times \Z/2\Z) \rtimes \Z.$$
\item if $g=a^2b^2c^2$, we have $d=2$, $m=3$, $D=1$.  In this case, $b^4=_G c^2=_G 1$ and the quotient
group is isomorphic to $\Z/4\Z \rtimes \Z$ with the action given by $aba^{-1}=b^3$.  This is a two-step-nilpotent
quotient of the Baumslag-Solitar group $BS(1,3)$ by introducing the relation $b^4=1$.
\end{enumerate}

We see that the quotient group $G$ collapses down to $\Z$
precisely if $\gcd(i,j)=1$.  
Namely, $c=_G 1$ in that case, so we have a quotient of
$\Z^2$ by a primitive vector.

\begin{corollary}
For one-relator quotients of the Heisenberg group,
$G=N_{2,2}/\nc g$,
$$
\Pr(G \cong \Z)= \frac{6}{\pi^2}\approx 60.8\% \ ; \qquad 
\Pr(G ~\hbox{\rm step 2, rank 2})= 1- \frac{6}{\pi^2}.
$$
\end{corollary}

Of course, if $g=c$, we have $\Z^2$, but this event occurs with probability zero.
If $\gcd(i,j)\neq 1$, then $G$ is two-step (thus non-abelian) and has
torsion.

\begin{proof}[Proof of theorem]

First, the $(i,j)=(0,0)$ case is very straightforward:
then $g=c^k$ and the desired expression for $G$ follows.

Below, we assume $(i,j)\neq(0,0)$, and by Lemma~\ref{basis-change-warmup}, without loss of generality, we will write $g=b^dc^m$.

Consider the normal closure of $b$, which is $\nc{b}=\langle b,c\rangle$.
This intersects trivially with $\langle a\rangle$, and $G=\nc{b}\langle a\rangle$.  Thus $G=\langle b,c\rangle\rtimes \langle a\rangle$.\\
Now in $H(\Z)$, we compute $\nc{g}=\langle b^dc^m,c^d\rangle\subset\langle b,c\rangle$. Thus 
$$\langle b,c\rangle \cong \groupmod{\Z^2}{\left\langle\vv dm, \vv 0d\right\rangle} .$$
We have the semidirect product structure 
$G\cong\groupmod{\Z^2}{\left\langle\vv dm, \vv 0d\right\rangle}\rtimes \Z$, where the action sends $\vv 10 \mapsto \vv 11$ and 
fixes $\vv 01$. 
Note that $c$ has order $d$ in $G$, 
and a simple calculation verifies 
that $b$ has order $d^2/D$, where $D=\text{gcd}(d,m)$. 
If we are willing to lose track of the action and just write the group up to isomorphism,
then we can perform both row and column operations on $\mm d0md$ to get $\mm {d^2/D}00D$,
which produces the desired expression.
\end{proof}

In fact, we can say something about quotients of $H(\Z)$ with arbitrary numbers of relators.
First let us define the {\em $K$-factor} $K(R)$ of a relator set $R=\{g_1,\dots,g_r\}$,
where  relator $g_1$ has the 
Mal'cev coordinates $(i_1,j_1,k_1)$, and similarly for  $g_2,\dots,g_r$.
Let $M=\left( \begin{array}{cccc} i_1 & i_2 & \dots & i_r \\ j_1 & j_2 & \dots & j_r \end{array}\right)$ and suppose its nullity (the dimension of its kernel) is $q$.  
Then let $W$ be a kernel matrix of $M$, i.e., an $r\times q$ matrix with 
rank $q$ such that $MW={\mathbf 0}$.  (Note that if $R$ is a random relator set,
then $q=r-2$, since the rank of $M$ is $2$ with probability one.)
Let $k=(k_1,\dots,k_r)$ be the vector of $c$-coordinates of relators, so that 
$kW\in \Z^q$.
Then $K(R):=\gcd(kW)$ is defined to be the gcd of those $q$ integers.

\begin{theorem}[Orders of Heisenberg quotients]\label{ord-heis-quot}
Consider the  group $G=H(\Z)/\nc{g_1,\ldots,g_r}$, 
where  relator $g_1$ has the 
Mal'cev coordinates $(i_1,j_1,k_1)$, and similarly for  $g_2,\dots,g_r$.
Let $d=\gcd(i_1,j_1,\dots,i_r,j_r)$; let $\Delta$ be the co-area of the lattice spanned
by the $\vv{i_\alpha}{j_\alpha}$ in $\Z^2$; and let $K=K(R)$ be the $K$-factor defined above.  
Then $c$ has order $\gamma=\gcd(d,K)$ 
in $G$ and 
$|G|=\Delta \cdot \gamma$.
\end{theorem}

\begin{proof}
Clearly $\Delta$ is the order of $\ab(G)=G/\langle c \rangle$.  So to compute the order of $G$,
we just need to show that the order of $c$ in $G$ is $\gamma$.  
Consider for which $n$ we can have $c^n \in \nc{g_1,\dots,g_r}$, i.e., 
$$c^n=\prod_{\alpha=1}^N w_\alpha \ g_\alpha^{\epsilon_\alpha} \ w_\alpha^{-1}$$
for arbitrary words $w_\alpha$ and integers $\epsilon_\alpha$.
First note that all commutators $[w,g_\alpha]$ are of this form, and that 
by letting $w=a$ or $b$, these commutators
can equal $c^{i_\alpha}$ or $c^{j_\alpha}$ for any $\alpha$, so 
$n$ can be an arbitrary multiple of $d$.

Next, consider the expression in full generality and note that $\A(c^n)=\vv 00$.
 Conjugation preserves weights, 
		so $\A(w_\alpha g_{\alpha}^{\epsilon_\alpha} w_\alpha^{-1})=\A(g_{\alpha}^{\epsilon_\alpha})=\epsilon_\alpha \A(g_{\alpha})=\epsilon_\alpha \vv{i_\alpha}{j_\alpha}$.
To get the two sides to be equal in abelianization, the $\epsilon_\alpha$ must record
a linear dependency in the $\vv{i_\alpha}{j_\alpha}$.
Finally we compute 
$$n=
\sum_\alpha \epsilon_\alpha(x_\alpha j_\alpha - y_\alpha i_\alpha)
+\sum_{\alpha<\beta} \epsilon_\alpha \epsilon_\beta i_\beta j_\alpha
-\sum_\alpha  i_\alpha j_\alpha\frac{ \epsilon_\alpha(\epsilon_\alpha-1)}{2}
+\sum_\alpha \epsilon_\alpha k_\alpha,$$
where $\vv{x_\alpha}{y_\alpha}=\A(w_\alpha)$.
We can observe that each of the first three terms  is a multiple of $d$ 
and the fourth term is an arbitrary integer multiple of $K$.
(To see this, note that the column span of $W$ is exactly the space of linear dependencies 
in the $\A(g_\alpha)$, so $\sum \epsilon_\alpha k_\alpha$ is a scalar product 
of the $k$ vector with something in that column span, and is therefore a multiple of $K$.)
Thus $n$ can be any integer combination of $d$ and $K$, as we needed to 
prove.
\end{proof}

We will include experimental data about the distribution of random Heisenberg
quotients in Section~\ref{sec-experiments}.

\section{Rank drop} \label{sec-rankdrop}

First, we establish that adding a single relator to a (sufficiently complicated) free nilpotent group 
does not drop the nilpotency class; the rank drops by one if the relator is primitive in abelianization
and it stays the same otherwise.  Furthermore, a single relator never drops
the step unless the starting rank was two.  This is a nilpotent version of Magnus's famous {\em Freiheitssatz}
(freeness theorem) for free groups \cite[Thm 4.10]{MKS}.  

\begin{theorem}[Nilpotent {\em Freiheitssatz}]
\label{injection-rank-drop} 
For any  $g\in N_{s,m}$ with $s\ge 2, m\ge 3$,
there is an injective homomorphism 
$$N_{s,m-1} \hookrightarrow N_{s,m}\slash \nc{g}.$$
This is an isomorphism if and only if 
$\gcd(A_1(g),\ldots,A_m(g))=1$.  

If $m=2$ the result holds with $\Z\hookrightarrow N_{s,2}\slash \nc g$.
\end{theorem}

\begin{proof} Romanovskii's 1971 theorem \cite[Thm 1]{Ro} does most of this.  In our language, the theorem says that 
if $A_m(g)\neq 0$, then $\langle a_1,\dots,a_{m-1} \rangle$ is a copy of $N_{s,m-1}$.  
This establishes the needed injection except in the case $g\in [N_{s,m},N_{s,m}]$, 
where $\A(g)$ is the zero vector.  In the $m=2$ case, any such $N_{s,2}/\nc g$
has abelianization $\Z^2$, so the statement holds.  For $m> 2$,
 one can apply an automorphism
so that $g$ is spelled with only commutators involving $a_m$.
Even killing all such 
commutators does not drop the nilpotency class because $m>2$ ensures that 
there are some Mal'cev generators spelled without $a_m$ in each level.
Thus in this case $\langle a_1,\dots,a_{m-1}\rangle\cong N_{s,m-1}$ still embeds.

It is easy to see that if $g$ is non-primitive in abelianization, then the rank of $\ab(N_{s,m}\slash \nc{g})$
is $m$, and so  the quotient nilpotent group has rank $m$ as well.
However, the image of Romanovskii's map has rank $m-1$, so it is not a surjection.

On the other hand, suppose $\ab(g)$ is a primitive vector.  Then the rank of the abelianized quotient
is $m-1$, and by Magnus's theorem (Theorem~\ref{magnus}) the rank of the nilpotent quotient is the same.  The group 
$G=N_{s,m}/\nc g$ is therefore realizable as a 
quotient of that copy of $N_{s,m-1}$.  Since the lower central series of $N_{s,m-1}$ has all 
free abelian quotients, any proper quotient would have smaller Hirsch length, and this contradicts Romanovskii's injection.  Thus relative primality implies that the injection is an isomorphism.
\end{proof}

Now we can use rank drop to analyze the probability of an abelian quotient for a free nilpotent group in the underbalanced, nearly balanced, and balanced cases (i.e., 
cases with the number of relators at most the rank).  

\begin{lemma}[Abelian implies rank drop for up to $m$ relators]\label{never-full-rank}
Let $G=N_{s,m}/\nc{R}$, where $R=\{ g_1,\ldots,g_r\}$ is a set of $r\le m$ random relators.
Suppose $s\ge 2$ and $m\ge 2$.  
Then $$\Pr(G ~\hbox{\rm abelian}\ \mid \ \rank(G)=m)=0.$$
\end{lemma}

\begin{proof}
	Suppose that $\rank(G)=m$ and $G$ is abelian.  
	We use the form of the classification of abelian groups 
	(Remark~\ref{fgabelclassification}) in which  
	$G\cong \oplus_{i=1}^m \Z/d_i\Z$, where $d_m \mid \dots \mid d_1$ so that 
	$d_1=\dots=d_q=0$ for $q=\dim(G)$, and we write 
	$\nc{\ab(R)}=\langle d_1e_1,\dots,d_me_m\rangle$ for a basis
	$\{e_i\}$ of $\Z^m$.
	Since $\rank(G)=m$, we can assume no $d_i=1$. We can lift the basis $\{e_i\}$ of 
	$\Z^m$ to a 
	generating set $\{a_i\}$ of $N_{s,m}$ by Magnus (Theorem~\ref{magnus}).	
	Note that the exponent of each generator in each relator is a multiple of 
	$d_m$.

	Next we show that we cannot kill a commutator in $G$ without dropping rank.  Let $b_1= [a_1, a_m]$. 
	We claim that 
	$b_1 \notin \nc{ g_1, \ldots, g_r }$.
	To do so, we compute an arbitrary element  
	$$ \prod_\alpha^n w_\alpha g_{\alpha}^{\epsilon_\alpha} w_\alpha^{-1} 
	\in \nc{g_1, \ldots, g_r}.$$ 
 Conjugation preserves weights, 
		so $\A(w_\alpha g_{i_\alpha}^{\epsilon_\alpha} w_\alpha^{-1})=\A(g_{i_\alpha}^{\epsilon_\alpha})=\epsilon_\alpha \A(g_{i_\alpha})$.
If the product is equal to $b_1$, then its $a$-weights are all zero.  
Now consider the $b$-weights.  For the product, the $b$-weights are the combination 
of the $b$-weights of the $g_\alpha$, modified by amounts created by commutation.  However, 
since all the $a$-exponents of all the $g_\alpha$ are multiples of $d_m$, we get
$$\sum \epsilon_i \A(g_i)=\left( \begin{smallmatrix} 0\\ 0 \\ \vdots \\ 0 \end{smallmatrix}\right),
\qquad 
\sum \epsilon_i \B(g_i)\equiv \left( \begin{smallmatrix} 1\\ 0 \\ \vdots \\ 0 \end{smallmatrix}\right)
\pmod{d_m},
$$
	where each $\epsilon_i$ is the sum of the $\epsilon_\alpha$ corresponding to $g_i$.
The second expression ensures that the $\epsilon_i$ are not all zero,
so the first equality is a linear dependence in the $\A(g_i)$, which has probability zero
since $r\le m$.
\end{proof}

\begin{theorem} (Underbalanced quotients are not abelian) 
Let $G=N_{s,m}/\nc{R}$, where $R=\{ g_1,\ldots,g_r\}$ 
is a set of $r\le m-2$ random relators $g_i$.  Then
$$\Pr(G \hbox{~\rm abelian})=0.$$
\end{theorem}

\begin{proof}  
Suppose that $G$ is abelian, and consider elements of $G$ as vectors in $\Z^m$ via the abelianization
map on $N_{s,m}$; in this way we get vectors $v_1=\A(g_1),\ldots,v_r=\A(g_r)$.
From the previous result we may assume $\rank(G)<m$.  
By  Lemma~\ref{gcd-rm}, 
 we can find a primitive vector $w$ as a linear combination of the $v_i$. 
 Then we apply the linear algebra lemma (Lemma~\ref{turbo-prop}) 
to extend $w$ appropriately so that $\spn(v_1,\ldots,v_r)=\spn(w,w_2,\ldots,w_r)$.
We can find a series of elementary row operations (switching, multiplication 
by $-1$, or addition) to get $(w,w_2,\ldots,w_r)$ from $(v_1,\ldots,v_r)$, and we lift these
operations to elementary Nielsen transformations (switching, inverse, or multiplication, 
respectively) in $N_{s,m}$ to get $(g',g'_2,\ldots,g'_r)$ from $(g_1,\ldots,g_r)$. 
Note that Nielsen transformations on a set of group elements preserve the subgroup they generate, 
so also preserve  normal closure. 
This lets us define $R'=\{g',g_2',\ldots,g_r'\}$ with $\nc{R'}=\nc R$. Since $g'$ has 
a weight vector $w$ whose coordinates are relatively prime, the Freiheitssatz (Theorem~\ref{injection-rank-drop})
ensures that $N_{s,m}/ \nc{g'}\cong N_{s,m-1}$. Thus we have
$G= N_{s,m-1}/\nc{g_2',\ldots,g_r'}$.

If $r\le m-2$, then iterating this argument $r-1$ times gives 
$G\cong N_{s,m-r+1}/\nc{g_r}$ for some new $g_r$, and $m-r+1\ge 3$.
Then we can apply Theorem~\ref{injection-rank-drop} to conclude that this quotient is not abelian, 
because its nilpotency class is $s>1$.
\end{proof}

\begin{proposition}[Cyclic quotients]
If $|R|=m-1$ or $|R|=m$, then abelian implies cyclic:
$$\Pr(G \hbox{~\rm cyclic} \mid G \hbox{~\rm abelian})=1.$$
\end{proposition}

\begin{proof}
Running the proof as above, we iterate the reduction $m-2$ times 
to obtain $G\cong N_{s,2}\slash \nc g$ or $N_{s,2}\slash \nc{g,g'}$.  

If $g$ (or any element of $\nc{g,g'}$) is primitive, then $G$ is isomorphic to $\Z$ or a quotient of $\Z$, i.e., $G$
is cyclic.

Otherwise, note that $N:=N_{s,2}$ has the Heisenberg group as a quotient 
($H(\Z)=N_1/N_3$).
If $G$ is abelian, then the corresponding quotient of $H(\Z)$ is abelian.  In the non-primitive case,
this can only occur if $c\in\nc{g,g'}$, which (as in the proof of Lemma~\ref{never-full-rank}) implies
$\A(g)=(0,0)$ (or 
a  linear dependency between $\A(g)$ and $\A(g')$).  
But by Corollary~\ref{chg-var-dist}, the changes of basis do not affect the probability
of linear dependency, so this has probability zero.
\end{proof}

\begin{corollary}
For nearly-balanced and balanced models, 
the probability that a random nilpotent group is abelian equals the probability 
that it is cyclic.

We reprise the table from \S\ref{sec-prelim}, recalling that values are truncated at four digits.
$$\begin{array}{c|ccccccc}
\Pr({\rm abelian}) & m=2 & m=3 & m=4 & m=10 & m=100 & m=1000 & 
m\to\infty \\
\hline
|R|=m-1 & .6079 & .5057 & .4672 & .4361 & .4357 & .4357 & .4357 \\
|R|=m & .9239 & .8842 & .8651 & .8469 & .8469 & .8469 & .8469 
\end{array}
$$
\end{corollary}

\medskip

\begin{corollary}[Abelian one-relator]\label{ab1r} For any step $s\ge 2$, 
$$\Pr(N_{s,m}/\nc{g} ~\hbox{\rm is abelian})=
\begin{cases}
6/\pi^2,& m=2\\
0,& m\ge 3.
\end{cases}
$$
\end{corollary}

Note that these last two statements agree for $m=2$, $|R|=m-1=1$.

\section{Trivializing and perfecting random groups}\label{sec-perfect}

In this final section, we first observe the low threshold for collapse of a random nilpotent
group, using the abelianization.  Then we will prove a statement lifting facts about 
random nilpotent group to facts about the LCS of classical random groups,
deducing that {\em random groups are perfect} with  the same threshold
again.

Recall that $T_{j,m}=\left\{ \left[a_{i_1},\ldots,a_{i_j}\right] : 
1\le i_1,\ldots,i_j \le m\right\}$ contains the basic nested commutators
with $j$ arguments.
In this section we fix $m$ and write $F$ for the free group,
so we can write $F_i$ for the groups in its lower central series.  
Similarly we write 
$N$ for $N_{s,m}$ (when $s$ is understood), and $T_j$ for $T_{j,m}$.
Note that $\nc{T_j}=F_j$, so $N=F/F_{s+1}$.  

For a random relator set $R\subset F$, we write  $\Gamma=F/\nc R$, 
$G=N/\nc R$, and $H=\Z^m/\langle R\rangle=\ab(\Gamma)=\ab(G)$,
using the abuse of notation from Lemma~\ref{string-arith} and treating $R$ as a set of strings from $F$ to be identified with its image in $N$ or $\Z^m$.
In all cases, $R$ is chosen uniformly from  words of length $\ell$
or $\ell-1$ in $F$.

First we need a result describing the divisibility properties of
the determinants of matrices whose columns
record the coordinates of random relators.  

\begin{lemma}[Common divisors of random determinants] \label{random-det}
Fixing $m$ and any $k> 10m$, let 
$d_\ell^{(k)}=\gcd(\Delta_{\ell,1},\dots,\Delta_{\ell,k})$ 
be the greatest 
common divisor of the determinants of $k$ random $m\times m$ matrices
all of whose columns are independently sampled from $\hat A_\ell$.  
Then, as $\ell\to\infty$,
$$\Pr(d_\ell^{(k)} = 1) \longrightarrow
\prod_{\text{primes } p}  1 - \left[ 1 - \left(1 - \frac{1}{p}\right)\left(1 - \frac{1}{p^2}\right) \cdots \left(1 - \frac{1}{p^m}\right)  \right]^k.$$
\end{lemma}

To prove this carefully requires dividing the primes into size ranges and verifying that only the small
primes ($p \le \log \log \ell$) contribute. See the Appendix \cite{appendix} for details.

The following theorem tells us that in sharp contrast to Gromov random groups,
where the number of relators  required to trivialize the group is exponential in $\ell$, 
even the slowest-growing unbounded functions, like $\log \log \log \ell$ or an 
inverse Ackermann function, suffice to collapse random abelian groups
and random nilpotent groups.

\begin{theorem}[Collapsing abelian quotients]\label{abelian-collapse}
For random abelian groups $H=\Z^m/\langle R\rangle$ with 
$|R|$ random relators, if $|R|\to\infty$ as a function of $\ell$, then 
$H=\{0\}$ with probability one (a.a.s.).  If $|R|$ is bounded as a function of $\ell$,
then there is a positive probability of a nontrivial quotient, both for each $\ell$
and asymptotically.
\end{theorem}

\begin{proof}
For a relator $g$ of length $\ell$, its image in $\Z^m$ is the random vector $\A(g)$,
which converges in distribution to a multivariate normal, 
as described in \S\ref{sec-prob}.  Furthermore, the image of this vector
in projection to $\Z/p\Z$ has entries that are asymptotically independently and uniformly distributed.
We will consider adding vectors to this collection $R$ until
they span $\Z^m$, which suffices to get $H=\{0\}$.
	
Choose $m$ vectors $v_1, \ldots v_m$ in $\Z^m$ at random. These vectors are a.a.s.\ 
$\R$-linearly independent, because
their distribution is normal and linear dependence is 
a codimension-one condition. Therefore they span 
a sublattice  $L_1\subset \Z^m$.
The covolume of $L_1$  (i.e., the volume of the fundamental domain) 
is $\Delta_{\ell,1}= \det(v_1, \ldots, v_m)$.  As we add more 
vectors, we refine the lattice.  Note that 
$\Delta_{\ell,1}=1$ if and only if $L_1=\Z^m$.
Similarly define $L_j$ to be spanned by $v_{(j-1)m+1},\dots,v_{jm}$ for $j=2,3,\dots$, 
and define $\Delta_{\ell,j}$ to be the corresponding covolumes.
	
Note that for two lattices $L,L'$, 
the covolume of the lattice $L\cup L'$ is always a common divisor of
the respective covolumes $\Delta,\Delta'$.  Therefore, the lattice $L_1\cup \dots\cup L_k$
has covolume $\le \gcd(\Delta_{\ell,1},\dots,\Delta_{\ell,k})$.  
From Lemma~\ref{random-det}, this gcd approaches 
$$
\prod_{\hbox{\scriptsize primes }p} 1- \left[ 1- \left(1-\frac 1p\right)
\left(1-\frac{1}{p^2}\right)\cdots\left(1-\frac{1}{p^m}\right) \right]^k, 
$$
as $\ell\to\infty$, and this in turn goes to 1 as $k\to\infty$.
(To see this, note that first applying a logarithm, then
exchanging the sum and the limit, gives an absolutely convergent sequence.)

On the other hand, it is immediate that 
for any finite $|R|$ there is a small but nonzero chance that all entries are even, say,
which would produce a nontrivial quotient group.
\end{proof}

Of course this also follows immediately from the statement in Lemma~\ref{lem-collapse}, because 
 $$\Pr(\spn\{v_1,\dots,v_r\}=\Z^m) = \frac{1}{\zeta(r-m+1)\cdots \zeta(r)}\longrightarrow 1$$  
for any fixed $m$ as $r\to\infty$.

We immediately get corresponding statements for random nilpotent groups
and standard random groups.
Recall that a group $\Gamma$ is called {\em perfect} if 
$\Gamma=[\Gamma,\Gamma]$; equivalently, if 
$\ab(\Gamma)=\Gamma/[\Gamma,\Gamma]=\{0\}$.

\begin{corollary}[Threshold for collapsing random nilpotent groups]
A random nilpotent group $G=N_{s,m}/\nc R$ is \aas trivial precisely
in those models for which $|R|\to\infty$ as a function of $\ell$.
\end{corollary}

\begin{corollary}[Random groups are perfect]
Random groups $\Gamma=F_m/\nc R$ are  \aas perfect precisely
in those models for which $|R|\to\infty$ as a function of $\ell$.
\end{corollary}

\begin{proof}
$\Z^m/\langle R\rangle=\{0\}  \iff  \ab(\Gamma)=\{0\} \iff \ab(G)=\{0\} \iff G =\{1\}$,
with the last equivalence from Theorem~\ref{magnus}.
\end{proof}

We have established that the collapse to triviality of a random nilpotent
group $G$ corresponds to the immediate stabilization of the lower central series
of the corresponding standard random group:
$\Gamma_1=\Gamma_2=\dots$
In fact, we can be somewhat more detailed about the relationship between 
$G$ and the LCS of $\Gamma$.

\begin{theorem}[Lifting to random groups]\label{lifting}
For $\Gamma=F_m/\nc R$ and  $G=N_{s,m}/\nc R$, they are related by
the isomorphism
$\Gamma/\Gamma_{s+1} \cong G$. 
Furthermore, the first $s$ of the successive LCS quotients of $\Gamma$ 
are the same as those in the LCS of $G$, i.e.,
$$\Gamma_i/\Gamma_{i+1} \cong G_i / G_{i+1} \qquad \hbox{for}~ 1\le i\le s.$$
\end{theorem}

\begin{proof}
Since homomorphisms respect LCS depth (Lemma~\ref{homomolcs}), 
the quotient map $\phi: F\to\Gamma$ gives  $\phi(F_{j})=\Gamma_{j}$ for all $j$. 
We have $$\Gamma/\Gamma_{s+1} 
\cong F/\nc{ R,F_{s+1}} \cong N/\nc R=G$$
by Lemma~\ref{string-arith} (string arithmetic).
	
From the quotient map 
$\psi: \Gamma\to G$, we get $\Gamma_i/\Gamma_{s+1}=\psi(\Gamma_i)=G_i$.  
Thus
$$\groupmod{G_i}{G_{i+1}} \cong
\groupmod{\Gamma_i/\Gamma_{s+1}}{\Gamma_{i+1}/\Gamma_{s+1}}
\cong \groupmod{\Gamma_i}{\Gamma_{i+1}}. \qedhere$$
\end{proof}

\begin{corollary}[Step drop implies LCS stabilization]
For  $G=N_{s,m}/\nc R$, if  $\step(G)=k<s=\step(N_{s,m})$, 
then the LCS of the random group $\Gamma$ stabilizes:
  $\Gamma_{k+1}=\Gamma_{k+2}=\dots$.
\end{corollary}

\begin{proof}
This follows directly from the previous result, since 
$\step(G)=k$ implies that  $G_{k+1}=G_{k+2}=1$, which 
means $G_{k+1}/G_{k+2}=1$. Since $k+1\le s$, we conclude that 
$\Gamma_{k+1}/\Gamma_{k+2}=1$.  Thus $\Gamma_{k+2}=\Gamma_{k+1}$,
and it follows by the definition of LCS that these also equal $\Gamma_i$ for 
all $i\ge k+1$.
\end{proof}

Thus, in particular, when a random nilpotent group (with $m\ge 2$) 
is abelian but not trivial, the 
corresponding standard random group has its lower central series 
stabilize after one proper step:
$$ \dots\Gamma_4= \Gamma_3 = \Gamma_2 \lhd \Gamma_1 = \Gamma$$
For instance, with balanced quotients of $F_2$ this happens about $92\%$ of the time.

In future work, we hope to further study the distribution of steps for random nilpotent
groups.

\section{Experiments}\label{sec-experiments}

\subsection{Multi-relator Heisenberg quotients}

The following table records the outcomes of 10,000 trials (1000 trials for each of the 
ten rows)
with relators of length $999$ and $1000$, for $G=H(\Z)/\nc R$.  Note that $|R|=2$ is the 
balanced case.

\begin{center}
\begin{tabular}{|r|c|c|c|c|c|c|c|c|}
\hline
  &$|G|=1$ & 
  \multicolumn{2}{|c|}{$G$ cyclic nontrivial} & 
  \multicolumn{2}{|c|}{$G$ abelian noncyclic }& 
  \multicolumn{2}{|c|}{$G$ nonabelian} & largest finite order\\
  &(trivial) & 
  \multicolumn{2}{|c|}{(rk 1)} & 
  \multicolumn{2}{|c|}{ (rk 2 step 1) }& 
  \multicolumn{2}{|c|}{(rk 2 step 2)} & \\  
\hline
  &&\cellcolor{gray!20}  {\small infinite} & {\small finite} &\cellcolor{gray!20} {\small infinite} & {\small finite} &
  \cellcolor{gray!20} {\small infinite} & {\small finite} &---\\
 \hline
 $|R|=1$ &0&\cellcolor{gray!20} 604&0&\cellcolor{gray!20} 0&0&\cellcolor{gray!20} 396&0&---\\
\hline
 2&1&\cellcolor{gray!20} 1&917&\cellcolor{gray!20} 0&0&\cellcolor{gray!20} 0&81&11178\\
\hline
 3&514&\cellcolor{gray!20} 0&467&\cellcolor{gray!20} 0&9&\cellcolor{gray!20} 0&10&717\\
\hline
 4&766&\cellcolor{gray!20} 0&228&\cellcolor{gray!20} 0&2&\cellcolor{gray!20} 0&4&104\\
\hline
 5&884&\cellcolor{gray!20} 0&116&\cellcolor{gray!20} 0&0&\cellcolor{gray!20} 0&0&7\\
\hline
 6&945&\cellcolor{gray!20} 0&55&\cellcolor{gray!20} 0&0&\cellcolor{gray!20} 0&0&4\\
\hline
 7&979&\cellcolor{gray!20} 0&21&\cellcolor{gray!20} 0&0&\cellcolor{gray!20} 0&0&3\\
\hline
 8&981&\cellcolor{gray!20} 0&19&\cellcolor{gray!20} 0&0&\cellcolor{gray!20} 0&0&3\\
\hline
 9&995&\cellcolor{gray!20} 0&5&\cellcolor{gray!20} 0&0&\cellcolor{gray!20} 0&0&2\\
\hline
 10&997&\cellcolor{gray!20} 0&3&\cellcolor{gray!20} 0&0&\cellcolor{gray!20} 0&0&2\\
\hline
\end{tabular} 
\end{center}

Note that the triviality column comports closely with the
probability of collapse described in \S\ref{sec-prelim}:
 $$\Pr(\spn\{v_1,\dots,v_r\}=\Z^2) = \frac{1}{\zeta(r-1)\cdot \zeta(r)},$$  
which predicts $0,0,506,769,891,948,975,988,994$, and $997$ trivial quotients.

\subsection{Finite nonabelian quotients of balanced presentations}

Because underbalanced ($|R|\le m-2$) and nearly-balanced ($|R|=m-1$) 
presentations necessarily produce infinite groups, while the overbalanced
case ($|R|\ge m+1$) often collapses the group, balanced presentations
are a good source for finite nonabelian quotients, as one sees in the table above.  Consider balanced quotients
of $H(\Z)$ for which the 
random relators have Mal'cev coordinates $(i_1,j_1,k_1)$ and $(i_2,j_2,k_2)$.  
Letting $\Delta=|i_1j_2-i_2j_1|$, the group is finite if and only if $\Delta>0$, in which 
case the order of the abelianization is $\Delta$.  Letting $d=\gcd(i_1,j_1,i_2,j_2)$,
we recall that $d=1$ implies a cyclic quotient, so the finite nonabelian case 
requires $\Delta>0$ and $d>1$.  
Having $\Delta>0$ implies that there are no nontrivial linear dependencies between
$\vv{i_1}{j_1}$ and $\vv{i_2}{j_2}$, so $K(R)=0$ (as in Theorem~\ref{ord-heis-quot}),
making the order of $c$ in the quotient group equal to $d$; since 
 $|\langle c\rangle|=d$, the quotient is nonabelian iff $d>1$.
Finally $|G|=\Delta\cdot d$, and we further note that 
$d^2\mid \Delta$, so $d^3$ divides the order of the group.  
  This means that the smallest possible orders of nonabelian quotients are $8,16,24,27,32,40,48,54,\dots$

With small-order groups, one can easily classify by isomorphism type, asking 
for instance how many of the order-eight nonabelian groups are isomorphic to
the quaternion group $Q=\{\pm 1, \pm i, \pm j, \pm k\}$ 
and how many to the  dihedral group $D_4$.  
However, since the expected magnitude of each of the entries in 
$\vv{i_1}{j_1}$ and $\vv{i_2}{j_2}$
is $\sqrt \ell$, the value of $\Delta$ and hence the expected size of these quotient groups is 
growing fast with $\ell$.  
Therefore to illustrate the distribution of random nilpotent groups that are small-order 
nonabelian, we consider $n=10,000$ trials with $\ell=9$ or $10$.
In this sample, $562$ of the quotients were finite nonabelian.

\smallskip

\begin{center}
\begin{tabular}{|l|c|c|c|c|c|c|c|c|c|c|c|c|c|c|c|c|c|}
\hline
order &8&16&24&27&32&40&48&54&56&64&72&80&81&88&96&125&216\\
\hline
frequency&130&138&65&41&45&32&35&24&9&18&9&3&6&1&2&2&2\\
\hline
\end{tabular}
\end{center}

\smallskip

\noindent Of the 130 groups in this sample of order eight, 
33  were isomorphic to $Q$, and the other 97 to $D_4$.

\subsection{One-relator Heisenberg quotients}

Finally, we use Lemma~\ref{basis-change-warmup} and 
Theorem~\ref{classification} to study the diversity of random infinite
groups appearing in the one-relator case.  Given a random relator 
whose Mal'cev coordinates are $(i,j,k)$, we first change variables as in the Lemma
to obtain coordinates $(0,d,m)$, where $d=\gcd(i,j)$ and $m=\frac{ij}{2d}(d-1) +k$.
In this presentation, as noted in the proof of the Theorem, 
$\mathop{\rm ord}(c)=d$ and $\mathop{\rm ord}(b)=d^2/D$, while 
 $a$ has infinite order.    Thus every word involving $a$ has infinite order; 
 on the other hand, $b$ and $c$ commute and so any word in those letters alone
 has order at most $d^2/D$ (note that this is divisible by $d$ because $D=\gcd(d,m)$).
Extracting information that is independent of presentation, we conclude that 
the order of the center is $d$ and the largest order of a torsion element is $d^2/D$.

We ran $n=20,000$ trials with $\ell=999$ or $1000$ and plotted the frequency of each
$(d^2/D,d)$ pair (Figure~\ref{plot}).  Besides the groups that are pictured, there were also
four occurrences of $(i,j)=(0,0)$ in the sample, with $k$ values $55,187,230,580$, that
are not pictured.
Because groups with distinct $(d^2/D,d)$ pairs must be non-isomorphic, our sample
 contains at least $202$ distinct groups (up to isomorphism).  

Recall from Remark~\ref{BS-case} that groups 
with $D=1$ are called Baumslag-Solitar type because they are isomorphic quotients of some $BS(1,q)$.
But $D=\gcd(d,m)$, and from the expression $m=\frac{ij}{2d}(d-1) +k$ we can 
note that $2d \mid ij$ if either $i$ or $j$ is even, in which case $\gcd(d,m)=\gcd(i,j,k)$.
(If both are odd, the situation splits into sub-cases depending on the $2$-adic valuations.)
This suggests heuristically that non-cyclic groups of Baumslag-Solitar type should 
occur with probability $\frac{1}{\zeta(3)} - \frac{1}{\zeta(2)}= .22398\dots$ 
The precise frequency of groups of this type in the sample was $4404$, or 
$22.02\%$.

\newgeometry{margin=1.5cm}
\begin{landscape}
\thispagestyle{empty}
\begin{figure}[ht]
\begin{center}
\begin{tikzpicture}
\begin{semilogxaxis}[log basis x={10}, width=20cm,
xlabel={largest torsion order ($d^2/D$)},
ylabel={size of center ($d$)},]
\addplot [domain=1:14000,gray] plot{x^(1/2)};
\addplot [domain=1:118.32,gray] plot{x};
\addplot [only marks,mark=*,mark size=1pt,gray,opacity=.6] coordinates 
{
(23, 23)
(24, 24)
(29, 29)
(31, 31)
(33, 33)
(41, 41)
(44, 22)
(45, 45)
(48, 24)
(60, 60)
(90, 30)
(99, 33)
(107, 107)
(117, 39)
(125, 25)
(128, 32)
(135, 45)
(160, 40)
(176, 44)
(180, 30)
(189, 63)
(300, 30)
(320, 40)
(400, 40)
(432, 36)
(675, 45)
(720, 60)
(867, 51)
(1352, 52)
(1458, 54)
(1568, 56)
(1728, 72)
(1764, 42)
(1800, 60)
(1922, 62)
(2048, 64)
(2178, 66)
(2187, 81)
(2304, 48)
(2312, 68)
(2883, 93)
(2916, 54)
(3025, 55)
(3249, 57)
(3528, 84)
(3600, 60)
(3721, 61)
(4225, 65)
(4356, 66)
(5184, 72)
(7056, 84)
(13689, 117)
(55,55)
};
\addplot [only marks,mark=*,mark size=2pt,gray,opacity=.6] coordinates 
{
(15, 15)
(18, 18)
(19, 19)
(20, 20)
(27, 27)
(37, 37)
(40, 20)
(63, 21)
(81, 27)
(112, 28)
(120, 60)
(144, 36)
(150, 30)
(484, 44)
(490, 70)
(648, 36)
(676, 52)
(722, 38)
(800, 40)
(882, 42)
(1250, 50)
(1600, 40)
(1936, 44)
(2450, 70)
(2500, 50)
(2704, 52)
(2738, 74)
(3136, 56)
(3364, 58)
(3481, 59)
(4489, 67)
(5329, 73)
(5625, 75)
};
\addplot [only marks,mark=*,mark size=3pt,gray,opacity=.6] coordinates 
{
(17, 17) (28, 28) (72, 24) (144, 24) (245, 35) (324, 36) (392, 28)
(512, 32) (605, 55) (900, 30) (968, 44) };
\addplot [only marks,mark=*,mark size=3.5pt,gray,opacity=.6] coordinates 
{  (13, 13)
 (32, 16)
 (52, 26)
 (54, 18)
 (96, 24)
 (294, 42)
 (507, 39)
 (1058, 46)
 (1296, 36)
 (2116, 46)
 (2209, 47)
 (2809, 53)
 (3844, 62)};
\addplot [only marks,mark=*,mark size=4pt,gray,opacity=.6] coordinates 
{ (14, 14) (196, 28) (288, 24) (363, 33) (450, 30) };
\addplot [only marks,mark=*,mark size=4.5pt,gray,opacity=.6] coordinates 
{(28, 14) (36, 18) (45,15)  (2601, 51)};
\addplot [only marks,mark=*,mark size=5pt,gray,opacity=.6] coordinates 
{ (16,16) (64,16) (80,20) (100,20) (192,24) 
 (243,27) (576, 24) (1521, 39) (1681, 41) };
\addplot [only marks,mark=*,mark size=5.5pt,gray,opacity=.6] coordinates 
{ (147, 21) (338, 26)  (676, 26) (1225, 35) (1444, 38) };
\addplot [only marks,mark=*,mark size=6pt,gray,opacity=.6] coordinates 
{(11,11) (12,12) (108,18) (784,28) (1024, 32) (1849, 43)};
\addplot [only marks,mark=*,mark size=8pt,gray,opacity=.6] coordinates 
{ (9, 9) (20, 10) (24, 12) (36, 12) (48, 12) (72, 12) (75, 15)
(128, 16) (162, 18) (200, 20) (242, 22) (400, 20) (441, 21)
(484, 22) (578, 34) (729, 27) (841, 29)
(961, 31) (1089, 33) (1156, 34) (1369, 37)
};
\addplot [only marks,mark=*,mark size=10pt,gray,opacity=.6] coordinates 
{  (7, 7) (8, 8) (10, 10) (16, 8) (27, 9) (50, 10) (98, 14) (100, 10)
(144, 12) (196, 14) (225, 15) (256, 16)
(324, 18) (361, 19) (529, 23) (625, 25) };
\addplot [only marks,mark=*,mark size=12pt,gray,opacity=.6] coordinates 
{ (6,6)(12,6) (32,8) (121,11) (169,13) (289, 17)};
\addplot [only marks,mark=*,mark size=14pt,gray,opacity=.65] coordinates 
{ (5,5) (8,4) (18,6) (36,6) (64,8) (81,9) };
\addplot [only marks,mark=*,mark size=16pt,gray,opacity=.7] coordinates 
{ (3,3) (4,4) (16,4) (25,5) (49,7)};
\addplot [only marks,mark=*,mark size=20pt,gray,opacity=.75] coordinates 
{ (9,3) };
\addplot [only marks,mark=*,mark size=24pt,gray,opacity=.8] coordinates 
{  (2,2) (4,2) };

\addplot  [mark=text,text mark={$\mathbb{Z}$}] coordinates {(1.5,113)} ;

\addplot [only marks,mark=*,mark size=1pt,gray,opacity=.5] coordinates
{(1.5,34)} node [black,opacity=1,right=1pt] {1};
\addplot [only marks,mark=*,mark size=8pt,gray,opacity=.5] coordinates
{(1.5,40)} node [black,opacity=1,right=8pt] {10};
\addplot [only marks,mark=*,mark size=15pt,gray,opacity=.65] coordinates
{(1.5,49)} node [black,opacity=1,right=15pt] {100};
\addplot [only marks,mark=*,mark size=22pt,gray,opacity=.8] coordinates
{(1.5,62)} node [black,opacity=1,right=22pt] {1000};

\addplot [only marks,mark=*,mark size=33pt,black,opacity=.6] coordinates
{(1.5,113)} node [black,opacity=1,right=33pt] {12208};

\draw (-.2,30) rectangle (.7,70);
\end{semilogxaxis}
\end{tikzpicture}
\end{center}
\caption{A semilog plot of $(d^2/D,d)$ in 20,000 random
1-relator quotients of the Heisenberg group with relator length $999$ or $1000$,
showing at least 202 mutually non-isomorphic groups.
Variously sized disks represent the number 
of occurrences of each $(d^2/D,d)$ value.
Since $D|d$, all possibilities lie between the curves $(d,d)$ and $(d^2,d)$.
Of these random groups, 
61\% are isomorphic to  $\mathbb Z$ and an additional 22\% are of 
Baumslag-Solitar
type ($D=1 \Rightarrow G \cong BS(1,q)/\nc{g}$ as in Remark~\ref{BS-case}), and thus lie along the lower curve $(d^2,d)$.\label{plot}}
\end{figure}
\end{landscape}
\restoregeometry

\subsection*{Funding}  This work 
was initiated in the Random Groups Research Cluster held at Tufts University
in Summer 2014, supported by the National Science Foundation [DMS-CAREER-1255442].

\subsection*{Acknowledgments} 
The co-authors had many stimulating and useful conversations in the course
of the work, including with 
Phil Matchett Wood, Nathan Dunfield, Dylan Thurston, 
Ilya Kapovich, Anschel Schaffer-Cohen, Rick Kenyon, and Larry Guth.
Special thanks to
 Keith Conrad 
and to Hannah Alpert for several insightful observations, and to two 
anonymous referees for close reading and very helpful suggestions.
Finally, Melanie Matchett Wood pointed out several issues in the preprint that merited more detail and care, which led to the creation of an appendix. We are extremely grateful for the feedback.

\newpage
\appendix
\section{Expanded information on arithmetic properties of random walks\\ \ \\
Moon Duchin, Meng-Che Ho, and Andrew S\'anchez}

\medskip

This appendix fills in details for some claims given above about arithmetic properties of random walks.
We will focus here on the simple random walk (SRW) on $\Z^m$, where the arguments will 
be spelled out in full, but also give 
indications of how to extend these arguments to non-backtracking simple random walk (NBSRW).
We will use the notation $\hat A_{\ell}$ for the $\Z^m$-valued random walk, as above,
and discuss the SRW and the NBSRW case separately in each argument below.

If  $E_\ell$ is an event that depends on a parameter $\ell$, we use the symbol $\Pr(E_\ell)$ for the probability for fixed $\ell$ 
and write $\Pbar(E_\ell):= \lim\limits_{\ell\to\infty} \Pr(E_\ell)$ for the asymptotic probability.
If $E$ is an event with respect to a matrix-valued random variable, 
we use the notation $\Pr\nolimits'(E)$ to denote the conditional probability of $E$ given that no matrix entries are zero.

We will analyze primes by their size relative to $\ell$, so we fix a 
small $\epsilon$ (say $0<\epsilon<\frac{1}{10}$) and define size ranges:
$$\Psm:=\{ p\le \log\log\ell \} \qquad \Pmed:=\{ \log\log\ell \le p \le \ell^{\frac 12 -\epsilon} \} \qquad
\Plg:=\{ \ell^{\frac 12 -\epsilon} \le p \le \ell^{m+1} \} \qquad \Phuge:= \{p\ge\ell^{m+1}\}.
$$

Recall that Lemma~\ref{arith-unif} provides that 
$\exists c_1,c_2>0 ~{\rm s.t.} \forall n<\ell^{\frac 12 -\epsilon}, \forall 1\le s\le m, \forall k_1,\dots,k_s,$  
$$\Pr\Bigl(\hat A_{\ell,i_1}\equiv k_1, \cdots, \hat A_{\ell,i_s}
\equiv k_s \mod n\Bigr)<
\frac{1}{n^s} +c_1e^{-c_2\ell^{2\epsilon}}.$$

Here we will carefully establish the following two results from above.

\setcounter{theorem}{16}
\begin{corollary}[Probability of primitivity]\label{rel-prime-coords}
For a random freely reduced word in $F_m$, the probability that it is primitive in abelianization
tends to  $1/\zeta(m)$, where $\zeta$ is the Riemann zeta function.  In particular, for $m=2$,  the probability is $6/\pi^2$.
\end{corollary}

\setcounter{theorem}{34}
\begin{lemma}[Common divisors of random determinants]  
Fixing $m$ and any $k> 10m$, let 
$d_\ell^{(k)}=\gcd(\Delta_{\ell,1},\dots,\Delta_{\ell,k})$ 
be the greatest 
common divisor of the determinants of $k$ random $m\times m$ matrices
all of whose columns are independently sampled from $\hat A_\ell$.  
Then, as $\ell\to\infty$,
$$\Pr(d_\ell^{(k)} = 1) \longrightarrow
\prod_{\text{primes } p}  1 - \left[ 1 - \left(1 - \frac{1}{p}\right)\left(1 - \frac{1}{p^2}\right) \cdots \left(1 - \frac{1}{p^m}\right)  \right]^k.$$
\end{lemma}

\subsection*{Probability of primitivity}

For SRW, $\hat A_{\ell,i}$ proceeds like a lazy simple random walk on $\Z$:  at each step, it advances 
left or right with probability $1/2m$, and otherwise it stands still.  A similar statement is true
for NBSRW, but the probabilities depend on the previous step.
As mentioned above,  classical central limit theorems tell us that $\hat A_{\ell,i}$ is asymptotically normally distributed, and this is true for the NBSRW case as well \cite{FH}.  
In this appendix we will sometimes use information 
about the rate of convergence of $\hat A_{\ell,i}$ to the Gaussian distribution.
For SRW, we have 
local central limit theorems (LCLT) which give upper bounds on the difference between 
the probability that $\hat A_{\ell,i}=x$ and the estimate derived from the Gaussian, in terms of $x$
and $\ell$ (see for instance Lawler-Limic, {\em Random Walk, A Modern Introduction}, Chapter 2).
For NBSRW, this is a folklore result that has not yet been written down, as far as we know.

\begin{app-lemma}[Divisibility of coordinate projections]\label{one-variable}
For every $m,n\ge 2$, $1\le s\le m$, 
and  $\ell\gg 1$, there is a  conditional probability bound given by 
$$\Pr\nolimits'(\hat A_{\ell,1}\equiv \dots \equiv \hat A_{\ell,s} \equiv 0 \mod n) < 1/n^s.$$
In particular,
$\Pr(\hat A_{\ell,i}\equiv 0 \mod n \mid \hat A_{\ell,i} \neq 0) < 1/n$ for any fixed $i$.
\end{app-lemma}

\begin{proof}  We give the detailed argument for $s=1$.
Let $p_\ell(x)=\Pr(\hat A_{\ell,i}=x)$.
This result will follow from monotonicity of the distribution of $\hat A_{\ell,i}$, i.e., 
$p_\ell(x)>p_\ell(x+1)$ for $x\ge 0$.
We proceed by induction on $\ell$.
 For $\ell=1$, we have $p_1(0)=\frac{2m-1}m$ and $p_1(1)=\frac{1}{4m}$, which establishes the base
 case.  For $\ell>1$, we have
 $$ p_\ell(x)=\frac{1}{2m} p_{\ell-1}(x-1)+\frac{m-1}{m} p_{\ell-1}(x) + \frac{1}{2m}p_{\ell-1}(x+1);$$
 $$ p_\ell(x+1)=\frac{1}{2m} p_{\ell-1}(x)+\frac{m-1}{m} p_{\ell-1}(x+1) + \frac{1}{2m}p_{\ell-1}(x+2).$$
Now we know that $\frac{m-1}m>\frac{1}{2m}$ (since $m\ge 2$), and this means
$$\frac{m-1}m p_{\ell-1}(x) + \frac{1}{2m} p_{\ell-1}(x+1) >
\frac{1}{2m} p_{\ell-1}(x) + \frac{m-1}{m} p_{\ell-1}(x+1),$$ 
since the LHS has a larger coefficient on the larger term.  This compares two of the terms of
$p_\ell(x)$ with two of the terms of $p_\ell(x+1)$, so it only remains to compare the remaining
terms.
Since $x\ge 0$, we have $|x-1|\le x+1$.  Thus, by repeatedly applying the inductive hypothesis, 
we have $p_{\ell-1}(x-1)>p_{\ell-1}(x+2)$, which completes the proof for all $\ell$.
In particular, we have shown:  if the positive integers $\Z_{>0}$ are partitioned into intervals 
$[kn+1,kn+n]$, then the farthest point in each interval from 0 (the value divisible by $n$) has the lowest
probability.

For NBSRW, we would need to inspect the  LCLT bounds to establish monotonicity rigorously, though it is intuitively clear for $\ell\gg 1$.

The argument for general $s$ runs along exactly the same lines:  $\Z_{>0}^m$ is cut up into 
$n\times \dots\times n$ boxes which are obtained as products of the intervals described above, 
then in each box, the 
point farthest from the origin (which satisfies the congruence condition in the statement of the lemma) has the lowest probability in the random walk.
\end{proof}

\begin{app-lemma}[Values of coordinate projections]\label{c}
There is a constant  $c$  such that  for any $\alpha\in \Z$ and any $i$ and $\epsilon>0$, 
$$\Pr(\hat A_{\ell,i}=\alpha)  <\frac{c}{\sqrt{\ell}} ~\hbox{\rm for}~ \ell\gg 1.$$ 
\end{app-lemma}

\begin{proof}
Bounding $\Pr(\hat A_{\ell,i}=\alpha)$ by a multiple of $\ell^{-1/2}$ follows from the standard local 
central limit theorem for SRW
 and could be extended to NBSRW from its LCLT.
\end{proof}

With this, we can establish the probability that a random relator is primitive in 
abelianization.
\begin{app-lemma}[Corollary 17]\label{fix of 17}
Let $\delta_\ell$ be the greatest common divisor of the entries of $\hat A_{\ell,i}$. Then
$$\Pbar(\delta_\ell>1) = 1 - \frac{1}{\zeta(m)}.$$
\end{app-lemma}
\begin{proof} Recall that for an event expressed in terms of a matrix-valued random variable,
$\Pr\nolimits'(E)$ denotes the conditional probability of $E$ given that the entries of the matrix are nonzero
(and this definition makes sense for vectors in particular).
Since
 $$\Pr(\delta_\ell>1)\le \Pr\nolimits'(\text{some prime divides}~\delta_\ell)+\Pr(\text{some entry of}~
\hat A_{\ell,i}~
\text{is zero}),$$ 
we have 
\begin{align*}\Pr(\text{some $p\in\Psm$ divides}~\delta_\ell)
\le \Pr(\delta_\ell>1)\le  
\Pr\nolimits'(\text{some $p\in\Psm$ divides}~\delta_\ell)&+\Pr\nolimits'(\text{some $p\in\Psm^c$ divides}~\delta_\ell)\\
&+ \Pr(\text{some entry is zero}).
\end{align*}

Recall the (well-known) fact that the product of all primes up to some $N$ is asymptotically 
$e^N$; this implies that $\displaystyle\prod_{p\in \Psm} p < \ell^{\frac 12-\epsilon}$.
Thus we can apply  Lemma~\ref{arith-unif} with $n=\displaystyle\prod_{\Psm} p$ to get 
asymptotic uniformity (and independence) for all of these primes at once.  
From this we get $\Pr\nolimits'(\text{some }p\in\Psm \text{ divides }\delta_\ell)\to 1 - \frac{1}{\zeta(m)}$, 
via the Euler product 
formula for the zeta function. By Lemma \ref{one-variable},
$$\Pr\nolimits'(\text{some $p\in\Psm^c$ divides}~\delta_\ell) < \sum\limits_{p\notin \Psm} \frac{1}{p^m}\longrightarrow 0,$$
where the  inequality is just the sum-bound $\Pr(\bigcup_i E_i) \le \sum_i \Pr(E_i)$ and it converges
to zero as the tail of a convergent sequence. 
Lastly, $\Pr(\text{some entry is zero})\to 0$ and the lemma follows. 
\end{proof}

\subsection*{Common divisors of random determinants}

We now build up a series of lemmas regarding divisibility with respect to our partition of the primes 
of determinants of random matrices $M_\ell$ with columns independently distributed by $\hat A_\ell$.  
We will refer to the upper left-hand $k\times k$ minor of such a matrix by $M_\ell^{(k)}$ 
(for $1\le k\le m$).

\begin{app-lemma}[Divisibility of determinants by small primes]\label{divdet-small}
Let 
$$\PP_m(p):= 1 - \left(1 - \frac{1}{p}\right)\left(1 - \frac{1}{p^2}\right) \cdots \left(1 - \frac{1}{p^m}\right).$$
There exist constants $c_1,c_2>0$ such that 
for all $\ell$, $k$, and  $p < \ell^{\frac{1}{2} - \epsilon}$ (i.e., $p\in \Psm \cup \Pmed$), 
$$\bigl| \ \Pr(p \mid d_\ell^{(k)}) -  \left[ \PP_m(p) \right]^k \ \bigr| 
< c_1 e^{-c_2\ell^{2\epsilon}}.$$ 
Furthermore,
$$\Pr(\hbox{\rm no }p\in\Psm \hbox{\rm ~  divides } d_\ell^{(k)}) =
\prod_{\Psm}\left( 1- \left[ \PP_m(p) \right]^k\right) + c_1e^{-c_2\ell^{2\epsilon}}.$$
\end{app-lemma}

\begin{proof}
The number of nonsingular matrices with $\mathbb F_p$ entries is
$$\bigl | GL_m(\mathbb F_p) \bigr| = 
\left( p^m-1 \right)
\left( p^m-p \right)
\cdots
\left( p^m-p^{m-1} \right)$$
out of $p^{m^2}$ total matrices \cite{robinson}, so the ratio of singular matrices 
is $\PP_m(p)$.
Thus the lemma follows from the fact that each entry of $M_\ell$ approaches a uniform distribution with the error term decaying exponentially fast in $\ell$.  Summing the error 
over the $km^2$ entries appearing in $k$ $m\times m$ matrices only worsens the 
constant $c_2$ that appeared in  Lemma~\ref{arith-unif}.

For the second statement we use the fact, noted in the last proof, that
probabilities are asymptotically uniform/independent for all primes in $\Psm$. 
The Chinese Remainder Theorem ensures that for any $m\times m$ matrices $A$
with entries in $\Z/p\Z$ and $B$ with entries in $\Z/q\Z$, there is a unique 
matrix $C$ with entries in $\Z/pq\Z$ that agrees with both in the respective projections.
Using this repeatedly, with $n=\displaystyle\prod_{\Psm} p$, we count that the 
number of matrices over $\Z/n\Z$ such that no $p\in\Psm$ divides the determinant
must equal $\prod_{\Psm} |GL_m(\mathbb F_p)|$.  The statement follows.
\end{proof}

To get a similar bound for large primes, we prove two lemmas on the divisibility of the 
determinants of the submatrices $M_\ell^{(k)}$, 
and then combine them for a bound that works on $\Plg$ and $\Phuge$.

\begin{app-lemma}[Divisibility of determinants by large primes]\label{divdet-large}
For $\epsilon$ as above,  
 there is a constant $c$ such that for sufficiently large $\ell$, any $1\le k\le m$, and 
 any prime $p \geq \ell^{1/2 - \epsilon}$ (i.e., $ p \in\Plg \cup \Phuge$), we have  
$$\Pr\nolimits'(\det M_{\ell}^{(k)} \equiv 0 \mod p) < \frac{c}{\ell^{\frac 12 -2\epsilon}}+\frac{c}{\sqrt{ \ell}}+\frac{1}{p},$$
 where $\Pr\nolimits'$ denotes conditional probability given that the matrix entries are nonzero.
It follows that there is a constant $c$ such that 
$ \Pr\nolimits'(p\mid \Delta_{\ell,i}) < c \ell^{2\epsilon-\frac{1}{2}}$ for $p\in\Plg\cup\Phuge$, and 
$ \Pr\nolimits'(p\mid \Delta_{\ell,i}) < c p^{\frac{4\epsilon-1}{2m+2}}$ for $p\in\Plg$.
\end{app-lemma}

\begin{proof}
For fixed $m$, we start with the $k=1$ case and raise $k$ one increment at a time to show that the probability that $M_\ell^{(k)}$ is divisible by $p$ is $\frac{2(k-1)c}{\ell^{\frac 12 -2\epsilon}}+\frac{(k-1)c}{\sqrt{ \ell}}+\frac{1}{p}$.
When $k=1$, this follows from Lemma~\ref{one-variable}.
Now suppose this is true for $M_\ell^{(k-1)}$.
Introduce the equivalence relation $A\sim B \iff a_{ij}=b_{ij}$ for all $(i,j)\neq (k,k)$; 
that is,  declare two $k\times k$ matrices equivalent if they agree 
in all entries except possibly the bottom right.
Then there is a constant $C_M$ for each matrix $M$ such that  
 $$\det A = a_{kk}\det N+C_M \qquad \forall A\in [M],$$  where $N$ is the upper-left-hand $(k-1) \times (k-1)$ minor.
Now if $p \nmid \det N$, then solving for $a_{kk}$ gives $(\det A-C_M)(\det N)^{-1}$ mod $p$, so at most $1/p$ of the $a_{kk}$ values in $\Z$ give a possible solution.
Thus there are at most
 $(2\ell^{1/2+\epsilon}/p)+1$ matrices $A\in [M]$ with determinant divisible by $p$ in this case, 
 and since $ \ell^{\frac 12 -\epsilon} \le p$ this has a conditional probability at most $(\frac{2\ell^{1/2+\epsilon}}{p}+1)\frac{c}{\sqrt{\ell}}<\frac{2c}{\ell^{1/2-2\epsilon}}+\frac{c}{\sqrt{ \ell}}$, given that the matrix falls in the equivalence class. 
 (The estimate comes from multiplying the number of matrices by the probability upper-bound for 
 each matrix; this bound is subject to an exponentially decaying error because the independence 
 is only asymptotic, but that is dominated by the $\sqrt \ell$.)
 
 By the $M_\ell^{(k-1)}$ hypothesis, the probability that $\det N $ is divisible by $p$ is 
 $<\frac{2(k-2)c}{\ell^{\frac 12 -2\epsilon}}+\frac{(k-2)c}{\sqrt{ \ell}}+\frac{1}{p}$, thus $\Pr\nolimits'(\det M_{\ell}^{(k)} \equiv 0 \mod p) < \frac{2(k-2)c}{\ell^{\frac 12 -2\epsilon}}+\frac{(k-2)c}{\sqrt{ \ell}}+\frac{1}{p} + \frac{2c}{\ell^{\frac 12 -2\epsilon}}+\frac{c}{\sqrt{ \ell}} = \frac{2(k-1)c}{\ell^{\frac 12 -2\epsilon}}+\frac{(k-1)c}{\sqrt{ \ell}}+\frac{1}{p}$. 
 After enlarging $c$, the first statement follows for $M_\ell^{(k)}$.
For the last statement, we want to combine these three terms.
Since $p > \ell^{\frac{1}{2} - \epsilon}$, we first observe  that
$\frac{1}{p} < \frac{1}{\ell^{1/2 -\epsilon}}  < \frac{1}{\ell^{1/2 - 2\epsilon}}$, and 
clearly $\frac{1}{\sqrt\ell}<\frac{1}{\ell^{1/2 - 2\epsilon}}$ as well.
Note that if $p \le \ell^{m + 1}$, then $\ell \ge p^{\frac{1}{m+1}}$, and we are done.
\end{proof}

\begin{app-lemma}[Nonsingularity]\label{nonsing}
 $\Pbar(\Delta=0) = 0$.
\end{app-lemma}

\begin{proof} The idea is that determinant zero is a codimension-one condition. To show it rigorously,
we prove the following stronger result:  for fixed $m$, we will show that 
$\Pbar(\det M_\ell^{(k)}=0)=0$ for each $1\le k\le m$.  For $k=1$, we note that $M_\ell^{(1)}=\hat A_{\ell,i}$, 
so the statement follows from Lemma~\ref{c}.  Let's show that if it is true for $M_\ell^{(k-1)}$, then 
it is true for $M=M_\ell^{(k)}$.
Let us write $q_\ell$ to denote the lower-right entry of $M_\ell^{(k)}$ and 
$\mu_{\ell}$ to denote the list of the other  $k^2-1$ entries $(M_{1,1}, \dots, M_{k,k-1})$.  
The induction hypothesis tells us the probability that $\det N=0$ tends to zero
for $N$ the upper left-hand  $k-1 \times k-1$ minor.
Assuming that minor is nonsingular, there is exactly one value 
of $q_\ell$ making $\det M_\ell^{(k)}=0$ for each $\mu$; call it $q(\mu)$.  
But, recalling that $0$ is the most likely
value for $q_\ell$ and that the different $\mu_\ell=\mu$ are disjoint events, we have
$$\Pr(\det M_\ell^{(k)}=0)=\sum_\mu \Pr(q_\ell=q(\mu)) \le \sum_\mu \Pr(q_\ell=0)=\Pr(q_\ell=0).$$
But $q_\ell$ is distributed like $\hat A_{\ell,i}$, so by Lemma~\ref{c}, this tends to zero. 
\end{proof}

\begin{app-lemma}[Lemma 35]  Fixing $m$ and any 
$k> 10m$, we have 
$$\Pbar(d_\ell^{(k)} = 1)= \prod_{\text{\rm primes } p}  1 - \left[ 1 - \left(1 - \frac{1}{p}\right)\left(1 - \frac{1}{p^2}\right) \cdots \left(1 - \frac{1}{p^m}\right)  \right]^k.$$
\end{app-lemma}

\begin{proof}
We'll break down the probability by dividing the primes into the 
size ranges $\Psm$, $\Pmed$, $\Plg$, and $\Phuge$.
As above, let $\PP_m(p):= 1 - \left(1 - \frac{1}{p}\right)\left(1 - \frac{1}{p^2}\right) \cdots \left(1 - \frac{1}{p^m}\right) $, and note that $\PP_m(p)\le \frac{2^m}p$ because there
are at most $2^m$ nonzero terms with denominators at least $p$.
We clearly have the following bounds:
\begin{align*}
\Pr(p|d_\ell^{(k)} ~\hbox{\rm for some}~ p\in \Psm) \ <  \ \Pr(d_{\ell}^{(k)} > 1)  \  &< \
\Pr\nolimits'(p|d_\ell^{(k)} ~\hbox{\rm for some}~ p\in \Psm) +
 \Pr\nolimits'(p|d_\ell^{(k)} ~\hbox{\rm for some}~ p\in \Pmed)\\
& {}+ \Pr\nolimits'(p|d_\ell^{(k)} ~\hbox{\rm for some}~ p\in \Plg)
+ \Pr\nolimits'(p|d_\ell^{(k)} ~\hbox{\rm for some}~ p\in \Phuge)\\
& {}+ \Pr(\text{some entry is zero}).
\end{align*}
We apply Lemma~\ref{divdet-small} and take a limit to get
$$\Pr(p \mid d_\ell^{(k)} \text{ for some }p\in \Psm) 
= 1 - \prod_{\Psm} \left( 1-\left[\PP_m(p)  \right]^k\right)+ O(e^{-\ell^{2\epsilon}}) \longrightarrow  1 - \prod_{\text{primes } p} \left( 1-\left[\PP_m(p)  \right]^k\right).$$
We have thus shown that 
$\Pbar(d_\ell^{(k)} >1) 
\ge 1 - \displaystyle\prod\limits_{\text{primes } p } \left( 1-\left[\PP_m(p)  \right]^k\right)$,
which implies that 
$$\Pbar(d_\ell^{(k)}=1)
\le \displaystyle\prod\limits_{\text{primes } p } \left( 1-\left[\PP_m(p)  \right]^k\right).$$
Note that $\Pr\nolimits'$ conditions on an event whose probability tends to $1$, thus
$\lim_{\ell\to\infty} \Pr\nolimits'(E)=\Pbar(E)$ if the limits exist.

To finish the theorem we must show the other four terms that bound $\Pr(d_\ell^{(k)} >1)$ limit to zero,
starting with the primes in $\Pmed$.  We have
$$\Pr(p \mid d_\ell^{(k)} \text{ for some }p\in \Pmed) < \sum_{\Pmed} \Pr(p|d_\ell^{(k)}) = \sum_{\Pmed} \left( \PP_m(p) + O(e^{-\ell^{2\epsilon}})\right)^k \longrightarrow 0, $$
where the $\PP_m(p)$ term appears because $p<\ell^{\frac 12 -\epsilon}$ means we can apply Lemma~\ref{divdet-small}.
To justify the convergence to zero, recall that $\PP_m(p)\le \frac{2^m}p$ and $k\ge 2$.

We now handle the case of $\Plg$, applying Lemma \ref{divdet-large} 
(and recalling that $k>10m$ and $\epsilon<\frac 1{10}$) to get
$$\Pr\nolimits'(p \mid d_\ell^{(k)} \text{ for some }p\in \mathcal P_3) \le \sum_{\Plg} \Pr\nolimits'(p \mid d_\ell^{(k)})=
\sum_{\Plg}\left( \Pr\nolimits'(p\mid \Delta_i) ^k \right)\le 
\sum_{\Plg} c \sdot p^{\textstyle\frac{4\epsilon-1}{2m+2}k}\le \sum_{\Plg} \frac{c}{p^2}.$$
Since the sum over all primes of $p^{-2}$ converges, this tail certainly converges to zero as $\ell\to\infty$.

In the range $\Phuge$, since all coordinates of the random walk vector are $\leq \ell$, 
we have $|\Delta_\ell|\leq m!\ell^m<\ell^{m+1}$ for $\ell\gg 1$. 
Since $\Delta_\ell=0$ is an asymptotically negligible event (Lemma~\ref{nonsing}), we have
 $$\Pr\nolimits'(p|d_\ell^{(k)} ~\hbox{\rm for some}~ p\in\Phuge)\longrightarrow 0,$$
because $\Pr(p|d_\ell^{(k)} ~\hbox{\rm for some}~ p\in\Phuge) 
=\Pr(\Delta_{\ell,1}=\dots=\Delta_{\ell,k}=0)\longrightarrow 0$, so $\Pr=\Pr'$.
Finally, the probability of a zero entry also goes to zero (Lemma~\ref{c}), which completes the proof.
\end{proof}

\end{document}

%% file: rank-drop-viz.tex
\newgeometry{margin=1.5cm}

\begin{figure}[ht]

\begin{tikzpicture}[scale=.55]

\begin{scope}[yshift=15cm]
\foreach \x/\z/\s in 
{	1/2/39, 	2/2/08, 	3/2/02, 	4/2/01, 	
	1/1/61, 	2/1/92, 	3/1/48, 	4/1/23, 	
	1/0/0, 	2/0/0, 	3/0/50, 	4/0/77, 	
	5/2/01, 	6/2/01, 	7/2/01, 	8/2/01,	9/2/01,	10/2/01,
	5/1/11, 	6/1/05, 	7/1/03, 	8/1/01,	9/1/01,	10/1/01,
	5/0/89, 	6/0/95, 	7/0/97, 	8/0/99,	9/0/99,	10/0/100
	}
{\draw [fill=blue!\s] (\x,\z) rectangle  (\x+1,\z+1) ;}
\node at (.6,-1/2)  {$|R|=$};
\foreach \p in {1,...,10}
{\node at (\p+.5,-.5) {$\p$};}
\node at (.75,2.5) [left] {$\rank(G)=2$};
\node at (.75,1.5) [left] {$1$};
\node at (.75,0.5) [left] {$0$};
\end{scope}

\begin{scope}[yshift=8cm]
\foreach \x/\z/\s in 
{	1/3/17, 	2/3/02, 	3/3/01, 	4/3/01, 	
	1/2/83, 	2/2/48, 	3/2/12, 	4/2/03, 	
	1/1/0, 	2/1/51, 	3/1/88, 	4/1/51, 	
	1/0/0, 	2/0/0, 	3/0/0, 	4/0/47, 	
	5/3/01, 	6/3/01, 	7/3/01, 	8/3/01,	9/3/01,	10/3/01,
	5/2/01, 	6/2/01, 	7/2/01, 	8/2/01,	9/2/01,	10/2/01,
	5/1/25, 	6/1/12, 	7/1/06, 	8/1/03,	9/1/01,	10/1/01,
	5/0/74, 	6/0/88, 	7/0/94, 	8/0/97,	9/0/98,	10/0/99
	}
{\draw [fill=blue!\s] (\x,\z) rectangle (\x+1,\z+1) ;}
\node at (.6,-1/2)  {$|R|=$};
\foreach \p in {1,...,10}
{\node at (\p+.5,-.5) {$\p$};}
\node at (.75,3.5) [left] {$\rank(G)=3$};
\node at (.75,2.5) [left] {$2$};
\node at (.75,1.5) [left] {$1$};
\node at (.75,0.5) [left] {$0$};
\end{scope}

\begin{scope}
\foreach \x/\z/\s in 
{	1/4/08, 	2/4/01, 	3/4/01, 	4/4/01, 	
	1/3/92, 	2/3/24, 	3/3/03, 	4/3/01, 	
	1/2/0, 	2/2/76, 	3/2/50, 	4/2/13, 	
	1/1/0, 	2/1/0, 	3/1/47, 	4/1/87, 	
	1/0/0, 	2/0/0, 	3/0/0, 	4/0/0, 	
	5/4/01, 	6/4/01, 	7/4/01, 	8/4/01,	9/4/01,	10/4/01,
	5/3/01, 	6/3/01, 	7/3/01, 	8/3/01,	9/3/01,	10/3/01,
	5/2/04, 	6/2/01, 	7/2/01, 	8/2/01,	9/2/01,	10/2/01,
	5/1/51, 	6/1/26, 	7/1/13, 	8/1/06,	9/1/03,	10/1/02,
	5/0/45, 	6/0/73, 	7/0/87, 	8/0/94,	9/0/97,	10/0/98
	}
{\draw [fill=blue!\s] (\x,\z) rectangle (\x+1,\z+1) ;}
\node at (.6,-1/2)  {$|R|=$};
\foreach \p in {1,...,10}
{\node at (\p+.5,-.5) {$\p$};}
\node at (.75,4.5) [left] {$\rank(G)=4$};
\node at (.75,3.5) [left] {$3$};
\node at (.75,2.5) [left] {$2$};
\node at (.75,1.5) [left] {$1$};
\node at (.75,0.5) [left] {$0$};
\end{scope}

\begin{scope}[yshift=-9cm]
\foreach \x/\z/\s in 
{	1/5/04,	2/5/01,	3/5/01,	4/5/01,
	1/4/96, 	2/4/10, 	3/4/01, 	4/4/01, 	
	1/3/0, 	2/3/90, 	3/3/26, 	4/3/03, 	
	1/2/0, 	2/2/0, 	3/2/74, 	4/2/53, 	
	1/1/0, 	2/1/0, 	3/1/0, 	4/1/44, 	
	1/0/0, 	2/0/0, 	3/0/0, 	4/0/0, 	
	5/5/01,	6/5/01,	7/5/01,	8/5/01,	9/5/01,	10/5/01,
	5/4/01, 	6/4/01, 	7/4/01, 	8/4/01,	9/4/01,	10/4/01,
	5/3/01, 	6/3/01, 	7/3/01, 	8/3/01,	9/3/01,	10/3/01,
	5/2/13, 	6/2/04, 	7/2/01, 	8/2/01,	9/2/01,	10/2/01,
	5/1/86, 	6/1/52, 	7/1/23, 	8/1/13,	9/1/06,	10/1/03,
	5/0/0, 	6/0/44, 	7/0/71, 	8/0/87,	9/0/94,	10/0/97
	}
{\draw [fill=blue!\s] (\x,\z) rectangle  (\x+1,\z+1) ;}
\node at (.6,-1/2)  {$|R|=$};
\foreach \p in {1,...,10}
{\node at (\p+.5,-.5) {$\p$};}
\node at (.75,5.5) [left] {$\rank(G)=5$};
\node at (.75,4.5) [left] {$4$};
\node at (.75,3.5) [left] {$3$};
\node at (.75,2.5) [left] {$2$};
\node at (.75,1.5) [left] {$1$};
\node at (.75,0.5) [left] {$0$};
\end{scope}

\begin{scope}[xshift=15cm,yshift=12cm]
\foreach \x/\z/\s in 
{	1/6/02,	2/6/01,	3/6/01,	4/6/01,
	1/5/98,	2/5/05,	3/5/01,	4/5/01,
	1/4/0, 	2/4/95, 	3/4/12, 	4/4/01, 	
	1/3/0, 	2/3/0, 	3/3/87, 	4/3/27, 	
	1/2/0, 	2/2/0, 	3/2/0, 	4/2/72, 	
	1/1/0, 	2/1/0, 	3/1/0, 	4/1/0, 	
	1/0/0, 	2/0/0, 	3/0/0, 	4/0/0, 	
	5/6/01,	6/6/01,	7/6/01,	8/6/01,	9/6/01,	10/6/01,
	5/5/01,	6/5/01,	7/5/01,	8/5/01,	9/5/01,	10/5/01,
	5/4/01, 	6/4/01, 	7/4/01, 	8/4/01,	9/4/01,	10/4/01,
	5/3/03, 	6/3/01, 	7/3/01, 	8/3/01,	9/3/01,	10/3/01,
	5/2/52, 	6/2/15, 	7/2/04, 	8/2/01,	9/2/01,	10/2/01,
	5/1/44, 	6/1/85, 	7/1/52, 	8/1/27,	9/1/13,	10/1/07,
	5/0/0, 	6/0/0, 	7/0/45, 	8/0/72,	9/0/87,	10/0/93
	}
{\draw [fill=blue!\s] (\x,\z) rectangle  (\x+1,\z+1) ;}
\node at (.6,-1/2)  {$|R|=$};
\foreach \p in {1,...,10}
{\node at (\p+.5,-.5) {$\p$};}
\node at (.75,6.5) [left] {$\rank(G)=6$};
\node at (.75,5.5) [left] {$5$};
\node at (.75,4.5) [left] {$4$};
\node at (.75,3.5) [left] {$3$};
\node at (.75,2.5) [left] {$2$};
\node at (.75,1.5) [left] {$1$};
\node at (.75,0.5) [left] {$0$};
\end{scope}

\begin{scope}[xshift=15cm,yshift=1.5cm]
\foreach \x/\z/\s in 
{	1/7/01,	2/7/01,	3/7/01,	4/7/01,
	1/6/99,	2/6/03,	3/6/01,	4/6/01,
	1/5/0,	2/5/97,	3/5/06,	4/5/01,
	1/4/0, 	2/4/0, 	3/4/94, 	4/4/13, 	
	1/3/0, 	2/3/0, 	3/3/0, 	4/3/87, 	
	1/2/0, 	2/2/0, 	3/2/0, 	4/2/0, 	
	1/1/0, 	2/1/0, 	3/1/0, 	4/1/0, 	
	1/0/0, 	2/0/0, 	3/0/0, 	4/0/0, 	
	5/7/01,	6/7/01,	7/7/01,	8/7/01,	9/7/01,	10/7/01,
	5/6/01,	6/6/01,	7/6/01,	8/6/01,	9/6/01,	10/6/01,
	5/5/01,	6/5/01,	7/5/01,	8/5/01,	9/5/01,	10/5/01,
	5/4/01, 	6/4/01, 	7/4/01, 	8/4/01,	9/4/01,	10/4/01,
	5/3/27, 	6/3/04, 	7/3/01, 	8/3/01,	9/3/01,	10/3/01,
	5/2/72, 	6/2/53, 	7/2/15, 	8/2/04,	9/2/01,	10/2/01,
	5/1/0, 	6/1/44, 	7/1/85, 	8/1/51,	9/1/27,	10/1/13,
	5/0/0, 	6/0/0, 	7/0/0, 	8/0/45,	9/0/72,	10/0/87
	}
{\draw [fill=blue!\s] (\x,\z) rectangle  (\x+1,\z+1) ;}
\node at (.6,-1/2)  {$|R|=$};
\foreach \p in {1,...,10}
{\node at (\p+.5,-.5) {$\p$};}
\node at (.75,7.5) [left] {$\rank(G)=7$};
\node at (.75,6.5) [left] {$6$};
\node at (.75,5.5) [left] {$5$};
\node at (.75,4.5) [left] {$4$};
\node at (.75,3.5) [left] {$3$};
\node at (.75,2.5) [left] {$2$};
\node at (.75,1.5) [left] {$1$};
\node at (.75,0.5) [left] {$0$};
\end{scope}

\begin{scope}[xshift=15cm, yshift=-10cm]
\foreach \x/\z/\s in 
{	1/8/01,	2/8/01,	3/8/01,	4/8/01,
	1/7/99,	2/7/01,	3/7/01,	4/7/01,
	1/6/0,	2/6/99,	3/6/03,	4/6/01,
	1/5/0,	2/5/0,	3/5/97,	4/5/06,
	1/4/0, 	2/4/0, 	3/4/0, 	4/4/94, 	
	1/3/0, 	2/3/0, 	3/3/0, 	4/3/0, 	
	1/2/0, 	2/2/0, 	3/2/0, 	4/2/0, 	
	1/1/0, 	2/1/0, 	3/1/0, 	4/1/0, 	
	1/0/0, 	2/0/0, 	3/0/0, 	4/0/0, 	
	5/8/01,	6/8/01,	7/8/01,	8/8/01,	9/8/01,	10/8/01,
	5/7/01,	6/7/01,	7/7/01,	8/7/01,	9/7/01,	10/7/01,
	5/6/01,	6/6/01,	7/6/01,	8/6/01,	9/6/01,	10/6/01,
	5/5/01,	6/5/01,	7/5/01,	8/5/01,	9/5/01,	10/5/01,
	5/4/13, 	6/4/01, 	7/4/01, 	8/4/01,	9/4/01,	10/4/01,
	5/3/87, 	6/3/28, 	7/3/04, 	8/3/01,	9/3/01,	10/3/01,
	5/2/0, 	6/2/71, 	7/2/53, 	8/2/14,	9/2/04,	10/2/01,
	5/1/0, 	6/1/0, 	7/1/43, 	8/1/86,	9/1/52,	10/1/27,
	5/0/0, 	6/0/0, 	7/0/0, 	8/0/0,	9/0/44,	10/0/72
	}
{\draw [fill=blue!\s] (\x,\z) rectangle  (\x+1,\z+1) ;}
\node at (.6,-1/2)  {$|R|=$};
\foreach \p in {1,...,10}
{\node at (\p+.5,-.5) {$\p$};}
\node at (.75,8.5) [left] {$\rank(G)=8$};
\node at (.75,7.5) [left] {$7$};
\node at (.75,6.5) [left] {$6$};
\node at (.75,5.5) [left] {$5$};
\node at (.75,4.5) [left] {$4$};
\node at (.75,3.5) [left] {$3$};
\node at (.75,2.5) [left] {$2$};
\node at (.75,1.5) [left] {$1$};
\node at (.75,0.5) [left] {$0$};
\end{scope}

\end{tikzpicture}

\vspace{.2in}

\begin{tikzpicture}[scale=.4]
\foreach \q in {3,...,14}
 { \foreach \z in {\q,...,15}
 	{\draw  [fill=blue!01] (17-\q,\z) rectangle  (18-\q,\z+1) ;}}

\foreach \q in {15,...,30}
 { \foreach \z in {3,...,15}
 	{\draw  [fill=blue!01] (\q,\z) rectangle  (\q+1,\z+1) ;}}

\foreach \q in {20,...,30}
 { \foreach \z in {1,2}
 	{\draw  [fill=blue!01] (\q,\z) rectangle  (\q+1,\z+1) ;}}

\foreach \x/\z in {1/15, 2/14, 2/15}
 	{\draw [fill=blue!01] (\x,\z) rectangle  (\x+1,\z+1) ;}

\foreach \x/\z/\s in 
{	1/14/100, 2/13/100, 3/12/100, 3/13/01, 4/11/100, 4/12/01,
5/10/100, 5/11/01, 6/9/100, 6/10/01, 7/8/100, 7/9/01,
8/7/99, 8/8/01, 9/6/98, 9/7/2, 10/5/97, 10/6/3, 11/4/93, 11/5/7, 12/3/86, 12/4/14,
13/2/72, 13/3/27, 14/1/43, 14/2/53, 14/3/4, 15/0/0, 15/1/85, 15/2/15,
16/0/43, 16/1/53, 16/2/4, 17/1/27, 17/2/01, 18/1/13, 18/2/01, 19/1/7, 19/2/01
	}
{\draw [fill=blue!\s] (\x,\z) rectangle  (\x+1,\z+1) ;}

\foreach \x/\z/\s in 
{	
17/0/72, 18/0/86, 19/0/93, 20/0/97, 21/0/98,
20/1/03
	}
{\draw [fill=blue!\s] (\x,\z) rectangle  (\x+1,\z+1) ;}

\foreach \x in {21,...,30}
 	{\draw [fill=blue!100] (\x,0) rectangle  (\x+1,1) ;}

\foreach \x in {1,...,14}
 {\foreach \z in {1,...,\x}
 	{\draw (15-\x,\z-1) rectangle  (16-\x,\z) ;}}

\node at (.4,-1/2)  {$|R|=$};
\foreach \p in {1,...,30}
{\node at (\p+.5,-.5) {$\p$};}
\node at (.75,15.5) [left] {$\rank(G)=15$};
\foreach \h in {0,...,14}
{\node at (.75,\h+.5) [left] {$\h$};}
\end{tikzpicture}
\caption{The empirical distribution of ranks in $\Z^m/\langle R\rangle$ for $m=2,3,4,5,6,7,8,15$.}

\end{figure}

\restoregeometry